\documentclass[11pt,reqno]{amsart}
\usepackage{setspace}
\usepackage{amsmath,amsfonts,amsthm,amsopn,amssymb}
\usepackage{amsmath,amssymb,latexsym,soul,cite,mathrsfs}
\usepackage{color,enumitem,graphicx}
\usepackage[colorlinks=true,urlcolor=blue,
citecolor=red,linkcolor=blue,linktocpage,pdfpagelabels,
bookmarksnumbered,bookmarksopen]{hyperref}
\usepackage[english]{babel}
\usepackage{enumitem}
\usepackage[left=2.5cm,right=2.5cm,top=2.8cm,bottom=2.8cm]{geometry}

\usepackage[hyperpageref]{backref}
\makeatletter
\newcommand{\leqnomode}{\tagsleft@true}
\newcommand{\reqnomode}{\tagsleft@false}
\makeatother

\numberwithin{equation}{section}

\newtheorem{thm}{Theorem}[section]
\newtheorem{lem}[thm]{Lemma}
\newtheorem{cor}[thm]{Corollary}
\newtheorem{Prop}[thm]{Proposition}
\newtheorem{Def}[thm]{Definition}
\newtheorem{Rem}[thm]{Remark}

\title[Critical heat equation with nonlocal interaction]{Asymptotic behavior of solutions for a critical heat equation with nonlocal reaction}

\author[J.\ Zhang]{Jian Zhang}
\author[J. \ Giacomoni]{Jacques Giacomoni}
\author[V.D.\ R\u adulescu]{{Vicentiu D. R\u adulescu}}
\author[M.\ Yang]{Minbo Yang}

\address{Jian Zhang  \newline\indent School of Mathematical Sciences, Zhejiang Normal University, \newline\indent
	Jinhua, Zhejiang, 321004, People's Republic of China}
\email{\href{mailto:email} {jzhang2021@zjnu.edu.cn}}

\address {Jacques Giacomoni\newline\indent LMAP, UMR E2S-UPPA CNRS 5142, B\^atiment IPRA, Avenue de l'Universit\'e F-64013 Pau, France}
\email{\href{mailto:email}{jacques.giacomoni@univ-pau.fr}}

\address{Vicentiu D. R\u adulescu \newline\indent Faculty of Applied Mathematics, AGH University of Krak\'ow, 30-059
	Krak\'ow, Poland \newline\indent
	Department of Mathematics, University of
	Craiova, 200585 Craiova, Romania \newline\indent
	School of Mathematical Sciences, Zhejiang Normal University, \newline\indent
	Jinhua, Zhejiang, 321004, People's Republic of China}
\email{ \href{mailto:email} {radulescu@inf.ucv.ro}}

\address{Minbo Yang  \newline\indent School of Mathematical Sciences, Zhejiang Normal University, \newline\indent
	Jinhua, Zhejiang, 321004, People's Republic of China}
\email{\href{mailto:email} {mbyang@zjnu.edu.cn}}

\begin{document}
	\maketitle
	
	\begin{abstract}
		In this paper, we consider the following nonlocal parabolic equation
		\begin{equation*}
			u_{t}-\Delta u=\left( \int_{\Omega}\frac{|u(y,t)|^{2^{\ast}_{\mu}}}{|x-y|^{\mu}}dy\right) |u|^{2^{\ast}_{\mu}-2}u,\ \text{in}\ \Omega\times(0,\infty),
		\end{equation*}
		where $\Omega$ is a bounded domain in $\mathbb{R}^{N}$, $0<\mu<N$ and $2^{\ast}_{\mu}=(2N-\mu)/(N-2)$ denotes the critical exponent in the sense of the Hardy-Littlewood-Sobolev inequality. We first introduce the stable and unstable sets for the equation and prove that the problem has a potential well structure. Next, we investigate the global asymptotic behavior of the solutions.
		In particular, we study the behavior of the global solutions that intersect neither with the stable set nor the unstable set.
		Finally, we prove that global solutions have $L^{\infty}$-uniform bound under some natural conditions.
		
		
	\end{abstract}
	
	\section{Introduction}
	In this paper, we are concerned with the following parabolic initial-boundary value problem:
	\begin{equation}\label{E}
		\left\{
		\begin{array}{lll}
			u_{t}-\Delta u=\left( \int_{\Omega}\frac{|u(y,t)|^{2^{\ast}_{\mu}}}{|x-y|^{\mu}}dy\right) |u|^{2^{\ast}_{\mu}-2}u,\ & x\in\Omega,\ t>0\\
			u(x,t)=0,\ & x\in\partial\Omega,\ t>0\\
			u(x,0)=u_{0}(x),\ & x\in\Omega,
		\end{array}
		\right.  \tag{$E$}
	\end{equation}
	where $\Omega$ is a bounded domain with smooth boundary in $\mathbb{R}^{N} (N\geq3)$, $0<\mu<N$ and $2^{\ast}_{\mu}=\frac{2N-\mu}{N-2}$ is the critical exponent in the sense of the Hardy-Littlewood-Sobolev inequality. Depending on the choice of different initial value $u_{0}$, we are interested in the asymptotic behavior of solutions.
	
	The nonlocal parabolic equation \eqref{E} has a very important background arising from a variety of physical, chemical and biological problems. It can be applied to nonlocal heat physics, where the nonlinear term represents the nonlocal source, see \cite{Lacey}. It can also be applied to the population dynamics model with nonlocal competitions, where the  nonlinear term models the behavior of individuals interacting with others 
	, see \cite{Gourley, Ou-Wu}, and references therein.
	
	In the last decades, the following parabolic initial-boundary value problem
	\begin{equation}\label{A1}
		\left\{
		\begin{array}{lll}
			u_{t}-\Delta u=\vert u\vert^{p-1}u,\ & x\in\Omega,\ t>0\\
			u(x,t)=0,\ & x\in\partial\Omega,\ t>0\\
			u(x,0)=u_{0}(x),\ & x\in\Omega
		\end{array}
		\right.
	\end{equation}
	was studied extensively. People were interested in the qualitative properties of solutions of equation (\ref{A1}) such as the sharp critical exponent value, the blow-up rate, blow-up profiles and existence of self-similar solutions, among the others.
	
	In \cite{Sattinger} Sattinger constructed the so called stable set, using the
	method of potential well, widely investigated in former contributions to study the existence of global solutions, see \cite{Ebihara-Nakao, Ikehata, Ikehata-Suzuki, Ishii, Lions1, Nakao-Ono, Otani, Tsutsumi1,Tsutsumi2} and references therein. In \cite{Levine,Payne-Sattinger}, Levine, Payne and Sattinger considered the blowing-up properties of solutions respectively. It is well known that the solutions of problem \eqref{A1} are global if the initial value belongs to the stable set whereas the solutions blow-up in a finite time if the initial value is in the unstable set. Then, the method of potential well was extensively considered to study the global existence and the blow-up pattern of solutions. Later, Brezis and Cazenave \cite{Brezis-Cazenave}, Weissler \cite{Weissler1,Weissler2} considered the well-posedness in Lebesgue space $L^{q}$ and showed that the exponent $q_{c}=N(p-1)/2$ is critical for the local existence of solutions to  equation \eqref{A1}. More precisely, for $u_{0}\in L^{q}$, there exists a unique classical solution on $[0,T)$ if $q>q_{c}$ and the problem \eqref{A1} is not well-posed when $q<q_{c}$. Moreover, the maximal time $T=T_{\max}$ can be chosen independently of $u_{0}$, and then either $T_{\max}=\infty$ or $\lim_{t\rightarrow T_{\max}}\Vert u(t)\Vert_{q}=\infty$ for $q>q_{c}$. If $q=q_{c}>1$, there exists $T>0$ such that \eqref{A1} possess a unique classical solution, but $T>0$ cannot be chosen independently of $u_{0}$. Furthermore, they obtained the global existence and uniqueness of the solution.  Concerning blow-up in finite time of solutions results to problem \eqref{A1}, we refer to Kaplan \cite{Kaplan}, Levine \cite{Levine,Levine-Payne}, Lions \cite{Lions2} and Lee, Ni \cite{Lee-Ni} where some criteria on the initial value $u_{0}$ are derived.	
	
	If $1<p<(N+2)/(N-2)$ and $\Omega\subset\mathbb{R}^{N}$ is a bounded domain, by using the potential well method, Ikehata and Suzuki \cite{Ikehata-Suzuki1996} obtained later that there exist two invariant sets (called stable and unstable sets respectively) describing the long time behavior of solutions to equation (\ref{A1}). Precisely, if the initial data $u_{0}$ belongs to the stable set, then the associated solution is global. However, if initial data $u_{0}$ belongs to the unstable set, then the associated solution blows-up in a finite time, and the phenomenon of blow-up in infinite time does not occur. For the blow-up of the associated energy at $T<+\infty$, see also \cite{Giga1,Quittner1}. If $p=(N+2)/(N-2)$, Tan \cite{Tan} proved the existence and asymptotic estimates of global solutions and the  behavior of blow-up solutions to equation \eqref{A1} with lower-energy initial values. Then, he also proved that the global $H^{1}_{0}$ solutions are classical global solutions. Ikehata and Suzuki \cite{Ikehata-Suzuki,Suzuki} further considered the behavior of solutions to problem (\ref{A1}) with $p=(N+2)/(N-2)$.
	If $\Omega$ is star-shaped, they proved that the solutions which intersect the unstable set at some time blow-up in finite time and global solutions which intersect the stable set at some time converge to $0$ uniformly as $t\rightarrow\infty$. Consequently, global solutions which intersect neither the stable nor the unstable sets blow-up in infinite time. Later, Ishiwata \cite{Ishiwata} showed that the solutions intersecting the stable set at some time exist globally in time.	
	Considering the global existence and blow-up patterns issues, we would like to point out different features arising in the case of critical exponent and in the case of the subcritical exponent. Since the embedding from $H^{1}_{0}(\Omega)$ into $L^{r}(\Omega)$ is compact for $r<2N/(N-2)$ and non-compact for $r=2N/(N-2)$, it is well know that the blow-up in infinite time does not occur and finite blow-up of the associated energy arises for the subcritical case. However, global unbounded solutions may exist for the critical case, see \cite{Ni-Sacks-Tavantzis}.
	
	Up to our knowledge, the following parabolic initial-boundary problem with nonlocal interaction:
	\begin{equation}\label{A2}
		\left\{
		\begin{array}{lll}
			u_{t}-\Delta u=\left(\int_{\Omega}\frac{|u(y,t)|^{p}}{|x-y|^{\mu}}dy\right) |u|^{p-2}u,\ & x\in\Omega,\ t>0\\
			u(x,t)=0,\ & x\in\partial\Omega,\ t>0\\
			u(x,0)=u_{0}(x),\ & x\in\Omega,
		\end{array}
		\right.
	\end{equation}
	was not investigated so far and there are only few works dealing with the asymptotic behavior of the solutions. Precisely, Liu and Ma \cite{Liu-Ma} considered the global existence and blow-up  in finite time of solutions to equation (\ref{A2}) with $\mu=N-2$ and $1<p<(N+2)/(N-2)$. Using the potential well method, they derived a threshold between the occurrence of global existence versus finite time blow-up of the solutions. In \cite{Zhang-Yang}, authors obtained the global existence of weak solutions and blow-up in finite time for problem \eqref{A2} with the critical exponent $p=2^{\ast}_{\mu}$ in the sense of the Hardy-Littlewood-Sobolev inequality.
	
	Recently, there are some results about the asymptotic behavior of the following nonlinear Schr\"odinger equation
	\begin{equation}\label{A3}
		iu_{t}+\Delta u+\left(\int_{\mathbb{R}^{N}}\frac{|u(y,t)|^{2^{\ast}}}{|x-y|^{N-2}}dy\right) |u|^{2^{\ast}-2}u=0.
	\end{equation}
	In \cite{Li-Miao-Zhang,Miao-Xu-Zhao1,Miao-Xu-Zhao2,Miao-Wu-Xu}, Miao et al. established the global well-posedness, scattering and blow-up of radial solutions for equation \eqref{A3} and characterized the global behavior of the finite time blow-up solutions.
	
	To the best of our knowledge, there is not any complete analysis of the asymptotic behavior of problem \eqref{E}, for instance, in terms of $L^{\infty}$-bound. And so, the purpose of the present paper is to study further the global asymptotic behavior of solutions to problem \eqref{E}. More precisely, we discuss the occurrence of global existence and blow-up in finite time of solutions under suitable conditions on initial values. Precisely we establish that if the initial value  intersects with the stable set, then the solution is global whereas it blows-up in finite time if the initial value intersects the unstable set. We also prove that the solutions blow-up in infinite time if the orbit intersects neither with the stable set nor the unstable set. In addition, we study the sufficient and necessary condition under which we get a uniform $L^{\infty}$-bound for global solutions. Here, we would like to mention that, for the critical case in the sense of the Hardy-Littlewood-Sobolev inequality, the functional $J_{\mu}$ (defined in \eqref{Energy} below) associated to problem \eqref{E} does not satisfy the Palais-Smale ($PS$ for short) condition (or $(PS)_{c}$ condition at level $c$). Nethertheless Yang et al. \cite{Du-Yang, Gao-Yang} proved the existence of ground states and classified the positive solutions to the associated stationary problem of \eqref{E}.
	
	It is well know that the local well-posedness plays an important role in studying the global asymptotic behavior of the solutions of the evolution equation. In \cite{Cazenave-Weissler}, Cazenave and Weissler already identified that there is a connection between $H^{s}$-theory and $L^{q}$-theory. To the best of our knowledge,  this is not clear for our problem \eqref{E} since the nonlinear term is of nonlocal type. Using the classical Hardy-Littlewood-Sobolev inequality, we shall verify the local well-posedness of problem \eqref{E} in $H^{1}_{0}(\Omega)$.
	In this regard, we will study problem \eqref{E} via the  corresponding integral equation below:
	\begin{equation}\label{PI}
		u(x,t)=(e^{t\Delta}u_{0})(x)+\int_{0}^{t}e^{(t-s)\Delta}\left\lbrace \left(\int_{\Omega}\frac{|u(y,s)|^{2^{\ast}_{\mu}}}{|x-y|^{\mu}}dy\right)|u(s)|^{2^{\ast}_{\mu}-2}u(s)\right\rbrace (x)ds \tag{$E1$}
	\end{equation}
	for any $t \in \left[0,\infty\right) $ and $x\in \Omega$, where $-\Delta$ denotes the Dirichlet Laplacian on $L^{2}(\Omega)$ and $\left\lbrace e^{t\Delta}\right\rbrace _{t>0}$ is the associated generated semigroup of contractions, defined for any $f\in L^2(\Omega)$ by
	\begin{equation*}\label{Heat Semigroup}
		(e^{t\Delta}f)(x):=(G(t,\cdot)\ast f)(x)
		=\int_{\Omega}G(t,x-y)f(y)dy,\ t>0,\ x\in\Omega
	\end{equation*}
	with
	\begin{equation*}
		G(t,x):=(4\pi t)^{-\frac{N}{2}}e^{\frac{\vert x\vert^{^{2}}}{4t}},\ t>0,\ x\in \Omega.
	\end{equation*}
	
	We recall the following basic inequality.
	\begin{lem}\label{L-HLS}
		{\rm (Hardy-Littlewood-Sobolev inequality, see \cite{Lieb-Loss})} Let $t,r>1$ and $0<\mu<N$ with $1/t+\mu/N+1/r=2$. For $f\in L^{t}(\mathbb{R}^{N})$ and $h\in L^{r}(\mathbb{R}^{N})$, there exists a sharp constant $C(t,N,\mu,r)$ independent of $f$ and $h$, such that
		\begin{equation}\label{A4}
			\int_{\mathbb{R}^{N}}\int_{\mathbb{R}^{N}}\frac{f(x)h(y)}{|x-y|^{\mu}}dxdy\leq C(t,N,\alpha ,r)\Vert f\Vert _{t}\Vert h\Vert _{r}.
		\end{equation}
		If $t=r=\frac{2N}{2N-\mu}$, then
		\begin{equation*}
			C(t,N,\mu,r)=C(N,\mu)=\pi^{\frac{\mu}{2}}\frac{\Gamma(\frac{N}{2}-\frac{\mu}{2})}{\Gamma(N-\frac{\mu}{2})}
			\{\frac{\Gamma(\frac{N}{2})}{\Gamma(N)}\}^{-1+\frac{\mu}{N}}.
		\end{equation*}
		In this case, the equality in (\ref{A4}) holds if and only if $f\equiv Ch$ and
		\begin{equation*}
			h(x)=A\left(\gamma^{2}+|x-a|^{2}\right)^{-(2N-\mu)/2}
		\end{equation*}
		for some $A\in\mathbb{C}$, $0\neq\gamma\in\mathbb{R}$ and $a\in\mathbb{R}^{N}$.
	\end{lem}
	
	From Lemma \ref{L-HLS} and using the Sobolev embedding, one has  for all $u\in D^{1,2}(\mathbb{R}^{N})$
	\begin{equation*}
		{\left(\int_{\mathbb{R}^{N}}\int_{\mathbb{R}^{N}}\frac{|u(x)|^{2^{\ast}_{\mu}}|u(y)|^{2^{\ast}_{\mu}}}
			{|x-y|^{\mu}}dxdy\right)^{\frac{N-2}{2N-\mu}}}
		\leq C(N,\mu)^{\frac{N-2}{2N-\mu}}\Vert u\Vert^{2}_{2^{*}},
	\end{equation*}
	where $C(N,\mu)$ is defined in Lemma \ref{L-HLS} and $2_{\mu}^{\ast}=\frac{2N-\mu}{N-2}$ is called the upper Hardy-Littlewood-Sobolev critical exponent.
	Denoting
	\begin{equation}\label{SHL}
		S_{H,L}:=\inf_{u\in D^{1,2}(\mathbb{R}^{N})\setminus\{0\}\}}\frac{\int_{\mathbb{R}^{N}}|\nabla u|^{2}dx}
		{\left(\int_{\mathbb{R}^{N}}\int_{\mathbb{R}^{N}}\frac{|u(x)|^{2^{\ast}_{\mu}}|u(y)|^{2^{\ast}_{\mu}}}
			{|x-y|^{\mu}}dxdy\right)^{\frac{N-2}{2N-\mu}}},
	\end{equation}
	we know from \cite[Lemma 1.2]{Gao-Yang} that the constant $S_{H,L}$ is achieved if and only if
	\begin{equation*}
		U=C\left(\frac{b}{b^{2}+|x-a|^{2}}\right)^{\frac{N-2}{2}},
	\end{equation*}
	where $C>0$ is a fixed constant, $a\in\mathbb{R}^{N}$ and $b\in(0,\infty)$ are parameters. Therefore,
	\begin{equation*}
		S_{H,L}=\frac{S}{C(N,\mu)^{\frac{N-2}{2N-\mu}}},
	\end{equation*}
	where $S$ is the best constant in the sense of Sobolev inequality. Moreover, for any open subset $\Omega\subset\mathbb{R}^{N}$, it is known that
	\begin{equation*}
		S_{H,L}(\Omega):=\inf_{u\in D^{1,2}_{0}(\Omega)\setminus\{0\}\}}\frac{\int_{\Omega}|\nabla u|^{2}dx}
		{\left(\int_{\Omega}\int_{\Omega}\frac{|u(x)|^{2^{\ast}_{\mu}}|u(y)|^{2^{\ast}_{\mu}}}
			{|x-y|^{\mu}}dxdy\right)^{\frac{N-2}{2N-\mu}}}=S_{H,L}
	\end{equation*}
	and $S_{H,L}(\Omega)$ is never achieved except when $\Omega=\mathbb{R}^{N}$. Following \cite{Sattinger,Payne-Sattinger}, we define the stable and unstable sets as
		\begin{equation}\label{W}
			W_{\mu}=\left\{u\in H^{1}_{0}(\Omega)\ |\ J_{\mu}(u)<m_{\mu},\ I_{\mu}(u)>0\right\}\cup\left\lbrace 0\right\rbrace
		\end{equation}
		and
		\begin{equation}\label{V}
			V_{\mu}=\left\lbrace u\in H^{1}_{0}(\Omega)\ |\ J_{\mu}(u)<m_{\mu},\ I_{\mu}(u)<0\right\rbrace
		\end{equation}
		respectively, where $I_{\mu}$ is the associated Nehari functional for problem \eqref{E} defined for any $u\in H^1_0(\Omega)$ by
\begin{equation*}
\displaystyle I_\mu(u):=\int_\Omega|\nabla u|^2dx-\int_\Omega\int_\Omega\frac{|u(x)|^{2^*_\mu}|u(y)|^{2^*_\mu}}{|x-y|^\mu}dxdy
\end{equation*}
and the associated critical level $m_\mu:=\frac{N-\mu+2}{2(2N-\mu)}S_{H,L}^{\frac{2N-\mu}{N-\mu+2}}$.

We first state below the notion of a (weak) solution to \eqref{E} on time-interval $[0,T)$ (with $0<T\leq\infty$) with initial data $u_0\in H^1_0(\Omega)$:
\begin{Def}\label{D-WS}
	Let $u_0\in H^1_0(\Omega)$. We call the function $u\in C([0,T); H^1_0(\Omega))$, satisfying $u(0)=u_0$, $u_t\in L^2((0,T)\times\Omega)$ and verifying the equation \eqref{E} in sense of distributions, a solution of \eqref{E}. The solution $u$ is said to be maximal if there is no solution extending $u$ outside the time-interval $(0,T)$. Such $T$ is denoted by $T_{\max}$. In addition, $u$ is said to be global if $T_{\max}=\infty$.
\end{Def}
We additionally give the notion of classical solution of \eqref{E}:
\begin{Def}
	Let $u_{0}\in H^{1}_{0}(\Omega)$ and $T\in (0,\infty]$, we say that  a weak solution $u\in C([0,T); H^1_0(\Omega))$ to \eqref{E} with initial data $u_0$ is a classical solution of \eqref{E} in $[0,T)$ if $u\in C^{2,1}(\Omega\times (0,T))\cap C(\bar{\Omega}\times (0,T))$  and $u$ satisfies \eqref{E} for $t\in (0,T)$ pointwisely.
\end{Def}
	We are now ready to state the main results. First, we give the existence of weak (maximal) solutions of \eqref{E}. 
	\begin{thm}\label{thm-existence}
				Suppose that $N\geq3$ and $0<\mu<\min\left\lbrace N,4\right\rbrace $. In addition, assume $\theta>0$ large enough such that $\alpha=\frac{1}{2}+\frac{N-2}{2^{\theta}}$ and $\beta=\frac{2^{\theta-1}(2N-\mu)-(N-2)(2N-\mu)}{2^{\theta}(N-\mu+2)}$ belong to $(\frac{1}{2},1)$. Then for each $u_{0}\in H^{1}_{0}(\Omega)$, there exists $T=T(\Vert u_{0}\Vert_{H^{1}_{0}})>0$ such that $u$ is a classical solution of problem \eqref{E} in $(0,T]$. Moreover wa have :
				\begin{equation}\label{bbb5}
					\Vert A^{\gamma}u(t)\Vert_{2}\leq\Vert Au(t)\Vert_{2}\leq Ct^{-\frac{1}{2}}\Vert u_{0}\Vert_{H^{1}_{0}}\ \text{for}\ t\in(0,T]
				\end{equation}
				for $\gamma=\alpha,\,\beta$ respectively, and
				\begin{equation}\label{bbb6}
					\Vert A^{\gamma}u_{t}(t)\Vert_{2}\leq Ct^{-(\gamma+\frac{1}{2})}\Vert u_{0}\Vert_{_{H^{1}_{0}}}\ \text{for}\ t\in(0,T]
				\end{equation}
				for any $\gamma\in[0,\min\{1-\alpha,1-\beta\})$. Furthermore, there exists $T_{\max}=T_{\max}(\|u_0\|_{H^1_0(\Omega)})\in (T,+\infty]$ such that $u$ can be extended as a maximal solution to \eqref{E} on the time interval $(0,T_{\max})$.
		\end{thm}
	
		Next, we deal with the existence of potential well structure and prove the blow-up in finite time of the solutions.
	\begin{thm}\label{T-Existence potential well}
		For any $0\leq\phi\in L^{\infty}\cap H^{1}_{0}(\Omega)$ and nontrivial, let $u\,:\, [0,T_{\max})\ni t\to H^1_0(\Omega)$ be a (maximal) solution to \eqref{E} with initial value $u_{0}=\lambda\phi$ for $\lambda>0$. Then there exists $0<\lambda_{1}\leq \lambda_{2}<\infty$ such that
		\begin{item}
			$(i)$ if $\lambda\in(0,\lambda_{1})$, then there exists $t_{0}\in[0,T_{\max})$ such that $u(t_{0})\in W_{\mu}$,\\
			$(ii)$ if $\lambda\in(\lambda_{2},\infty)$, then there exists $t_{0}\in[0,T_{\max})$ such that $u(t_{0})\in V_{\mu}$,\\
			$(iii)$ if $\lambda\in(\lambda_{1},\lambda_{2})$, then the orbit $\{u(t), 0\leq t< T_{\max} \}$ does not intersect with $W_{\mu}\cup V_{\mu}$.
		\end{item}
	\end{thm}
	
	\begin{thm}\label{T-Blow up}
		Let  $u\,:\, [0,T_{\max})\ni t\to H^1_0(\Omega)$ be a maximal solution such that $u(t_{0})\in V_{\mu}$ for some $t_{0}\in[0,T_{\max})$. Then $T_{\max}<+\infty$,
		\begin{equation*}
			u(t)\in V_{\mu},\ \forall t\in[t_{0},T_{\max})
		\end{equation*}
		and
		\begin{equation*}
			\Vert u(t)\Vert_{\infty}, \|u(t)\|_{H^1_0(\Omega)}\rightarrow\infty\ \text{as}\ t\rightarrow T_{\max}^-.
		\end{equation*}
	\end{thm}
	Now, we are going to prove the global existence of solutions by using the scaling arguments. Let $u\, :\, [0,T_{\max})\times\Omega\ni (t,x)\to \mathbb R$ be a solution of problem \eqref{E} and $\lambda>0$. For any $x_{0}\in\mathbb{R}^{N}$ and $t_{0}\in\mathbb{R}^+$, define
	\begin{equation}\label{Scaling}
		v(y,s):=\lambda^{-\frac{N-2}{2}}u(x,t),
	\end{equation}
	where $y=\lambda(x-x_{0})$ and $s=\lambda^{2}(t-t_{0})$.
	Then we show that the change of variables given in (\ref{Scaling}) makes invariant for problem \eqref{E} and the energy functional $J_{\mu}$ (see Proposition \ref{P-Scaling properties} below). We state the following results about the asymptotics of the global solutions.
	
	\begin{thm}\label{T-Global existence}
		Let $u\,:\, [0,T_{\max})\ni t\to H^1_0(\Omega)$ be a solution such that $u(t_{0})\in W_{\mu}$ for some $t_{0}\in[0,T_{\max})$. Then $T_{\max}=+\infty$. Moreover,
		\begin{equation*}
			u(t)\rightarrow0\ \text{in}\ L^{r}(\Omega),\ \forall r\in[2,\infty],
		\end{equation*}
		\begin{equation*}
			u(t)\rightarrow0\ \text{in}\ H_{0}^{1}(\Omega)
		\end{equation*}
		and $J_\mu(u(t))\rightarrow0$ as $t\rightarrow+\infty$.
	\end{thm}
	
	Next, we are interested in the behavior of the global solutions that intersect with neither the stable set nor the unstable set. Suppose $T_{\max}=+\infty$ and let
	\begin{equation}\label{C_{0}}
		C_{0}\equiv\frac{22^{\ast}_{\mu}}{2^{\ast}_{\mu}-1}\lim_{t\rightarrow\infty}J_{\mu}(u(t)).
	\end{equation}
	Clearly, $C_{0}\geq0$. Indeed, for a solution global in time, we must have $J(u(t))\geq0$ for all $t\geq0$ by the proof of \cite[Theorem 1.5]{Zhang-Yang} (or see \cite{Tsutsumi1}), and so $\lim_{t\rightarrow\infty}J_{\mu}(u(t))$ exists and is nonnegative. Thus, $C_{0}\geq0$. Then we state the following result of blow-up in infinite time. Here, we point out that the star-shapedness of domain $\Omega$ in the following Theorem excludes stationary nonnegative solutions of problem \eqref{E} to be nontrivial, which plays an important role for the long-time behaviour of the weak solutions.
	
	\begin{thm}\label{T-Blow up infinite}
		Let $u$ be a global solution of \eqref{E} with nonnegative initial data $u_0$ and suppose that $N=3$ and $\Omega$ is a star-shaped domain. Then, the following statements are equivalent:
		\begin{item}
			$(i)$ $C_{0}>0$,\\
			$(ii)$ $J(u(t))\geq m_{\mu}$ for all $t\in[0,+\infty)$,\\
			$(iii)$ $u(t)\not\in W_{\mu}\cup V_{\mu}$ for all $t\in[0,+\infty)$,\\
			$(iv)$ $\displaystyle\lim_{t\rightarrow\infty}\Vert u(t)\Vert_{\infty}=+\infty$.
		\end{item}
	\end{thm}
	
	Consequently, using Theorems \ref{T-Existence potential well}, \ref{T-Global existence} and Theorem \ref{T-Blow up infinite}, we may determine the asymptotic behavior of the solutions in the following corollary.
	
	\begin{cor}\label{Corollary}
		$(i)$ 
		Let $0\leq\phi\in L^{\infty}\cap H^{1}_{0}(\Omega)$, $u_{0}=\lambda\phi$ and $u$ be the solution to \eqref{E} with initial data $u_0$. Then, $T_{\max}=\infty$ and  $u(t)\rightarrow0$ in $H^{1}_{0}(\Omega)\cap L^{\infty}(\Omega)$ as $t\rightarrow\infty$ if $\lambda\in(0,\lambda_{1})$ and $\Vert u(t)\Vert_{\infty}, \|u(t)\|_{H^1_0(\Omega)}\rightarrow\infty$ as $t\rightarrow T_{\max}^-<+\infty$ if $\lambda\in(\lambda_{2},\infty)$.\\
		$(ii)$ 
		 If $\lambda\in(\lambda_{1},\lambda_{2})$, then $u(t)$ blow-up in infinite time, i.e. $\lim_{t\rightarrow\infty}\Vert u(t)\Vert_{\infty}=+\infty$.
	\end{cor}
	There are a lot of works devoted to the existence of $L^{\infty}$-uniform bounds for the global solutions, see \cite{Brezis-Cazenave,Cazenave-Lions,Fila-Souplet-Weissler,Giga1, Quittner1,Quittner2} and references therein. It is well known that every global solution of equation \eqref{A1} has an $L^{\infty}$-global bound for $1<p<(N+2)/(N-2)$. However, the conclusion for the case $p=(N+2)/(N-2)$ is much more delicate. Indeed, due to the non compactness of the nonlinear term, the method used for the subcritical case is not further valid for the critical one and there exists an unbounded global solution in this case. In \cite{Ishiwata2007}, Ishiwata gave some sufficient and necessary conditions for the $L^{\infty}$-uniform boundness of  global solution when $p=(N+2)/(N-2)$. In this paper, we aim to classify the solutions of problem \eqref{E} which  are uniform bounded through some suitable sufficient and necessary conditions. More precisely, we shall describe the relation between the existence of $L^{\infty}$-uniform bounds for global solutions to problem \eqref{E} and the compactness property of the associated energy functional $J_{\mu}$.
		Denoting
	\begin{equation}\label{c}
		c:=\lim_{t\rightarrow\infty}J_{\mu}(u(t)),
	\end{equation}
	we recall
	\begin{Def}\label{D-PS}
		{\rm ($(PS)_{c}$ sequence and $(PS)_{c}$ condition)} Let $c\in\mathbb{R}$.\newline
		$(i)$ A sequence $\left\lbrace u_{n}\right\rbrace $ is a $(PS)_{c}$ sequence in $H^{1}_{0}(\Omega)$ for $J_{\mu}$ if $J_{\mu}(u_{n})=c+o(1)$ and $J_{\mu}^{\prime}(u_{n})=o(1)$ strongly in $(H^{1}_{0}(\Omega))^{\ast}$ as $n\rightarrow\infty$, and $c$ is said to be $(PS)$ value in $H^{1}_{0}(\Omega)$ for $J_{\mu}$;\newline
		$(ii)$ $J_{\mu}$ satisfies $(PS)_{c}$ condition in $H^{1}_{0}(\Omega)$ if any $(PS)_{c}$ sequence in $H^{1}_{0}(\Omega)$ for $J_{\mu}$ contains a convergent subsequence.
	\end{Def}
	\begin{Def}\label{D-GB}
	{\rm ($L^{\infty}$-uniformly bounded)} A global solution $u$ is said to be $L^{\infty}$-uniformly bounded if
		\begin{equation*}
			\sup_{t\in[0,+\infty)}\Vert u\Vert_{\infty}<\infty.
		\end{equation*}
		Otherwise, $u$ is an unbounded global solution.
	\end{Def}
	We have the $L^{\infty}$-uniform boundedness of the global solutions.
	\begin{thm}\label{T-Global bounds}
		 For any global solution $u$ of problem \eqref{E}, the following statements are equivalent:
		\begin{item}
			$(i)$ The energy $J_{\mu}$ of problem \eqref{E} satisfies the $(PS)_{c}$-condition,\\
			$(ii)$ $u$ is $L^\infty$-uniformly bounded.
		\end{item}
	\end{thm}
	
	\begin{cor}\label{C-Global bounds}
		Let $u$ be a global solution of problem \eqref{E} such that $c<\frac{N-\mu+2}{2(2N-\mu)}S^{\frac{2N-\mu}{N-\mu+2}}_{H,L}$, where $c$ is given in (\ref{c}).
		Then $u$ is $L^{\infty}$-uniformly bounded.
	\end{cor}
	
	The rest of this paper is organized as follows. In Section 2, we introduce the properties of stable and unstable sets and give the result of local well-posedness in $H^{1}_{0}(\Omega)$ from which the proof of Theorem \ref{thm-existence} follows. In Section 3, we prove Theorems \ref{T-Existence potential well}-\ref{T-Global existence}. In section 4, we investigate the proof of Theorem \ref{T-Blow up infinite} and give the long time behavior of solutions to problem \eqref{E}. In Section 5, we prove Theorem \ref{T-Global bounds} and verify that any global solution is $L^\infty$-uniformly bounded under suitable conditions from which Corollary \ref{C-Global bounds} follows.

	\section{Preliminaries}
	In this section, we recall some basic properties of stable and unstable sets.
	First, for convenience, we set for $u\in H^1_0(\Omega)$
	\begin{equation*}
		A(u):=\int_{\Omega}\vert \nabla u\vert^{2}dx, \ \ \ B(u):=\int_{\Omega}\int_{\Omega}\frac{|u(x)|^{2^{\ast}_{\mu}}|u(y)|^{2^{\ast}_{\mu}}}{|x-y|^{\mu}}dxdy,
	\end{equation*}
	and define
	\begin{equation}\label{Energy}
	J_\mu(u)=\frac{1}{2}A(u)
		-\frac{1}{22^{\ast}_{\mu}}B(u),
	\end{equation}
	which acts as a Lyapunov type functional along trajectories associated to problem \eqref{E}. Setting
	\begin{equation*}
		\tilde m_{\mu}:=\inf_{u\in H^{1}_{0}(\Omega)\setminus\{0\}}\sup_{\theta>0}J_{\mu}(\theta u),
	\end{equation*}
	we have that
	\begin{equation}\label{Potential Deep}
		\tilde m_{\mu}=m_\mu=\frac{N-\mu+2}{2(2N-\mu)}S^{\frac{2N-\mu}{N-\mu+2}}_{H,L}.
	\end{equation}
	Indeed, by a simple calculation, we derive that
	\begin{equation*}
		\sup_{\theta>0}J_{\mu}(\theta u)=J_{\mu}(\bar{\theta}u)
		=\frac{N-\mu+2}{2(2N-\mu)}\left\lbrace \frac{A(u)}{\left(B(u)\right)^{\frac{N-2}{2N-\mu}} }\right\rbrace ^{\frac{2N-\mu}{N-\mu+2}},
	\end{equation*}
	where $\bar{\theta}:=\left( \frac{A(u)}{B(u)}\right)  ^{\frac{N-2}{2(N-\mu+2)}}$.
	Therefore, we have
	\begin{equation*}
		\tilde{m}_{\mu}=\frac{N-\mu+2}{2(2N-\mu)}\inf_{u\in H^{1}_{0}(\Omega)\setminus\{0\}}\left\lbrace \frac{A(u)}{\left(B(u)\right)^{\frac{N-2}{2N-\mu}} }\right\rbrace ^{\frac{2N-\mu}{N-\mu+2}}=\frac{N-\mu+2}{2(2N-\mu)}S^{\frac{2N-\mu}{N-\mu+2}}_{H,L},
	\end{equation*}
	where $S_{H,L}$ is the sharp constant of the Hardy-Littlewood-Sobolev inequality.
	
From \eqref{V} and \eqref{W} we have:
	\begin{equation}\label{Nehari-Energy}
		I_{\mu}(u):=A(u)-B(u).
	\end{equation}
	Set
	\begin{equation*}
		S_{\mu}:=\left\lbrace u\in H^{1}_{0}(\Omega)\ |\ -\Delta u=\left( \int_{\Omega}\frac{|u(y)|^{2^{\ast}_{\mu}}}{|x-y|^{\mu}}dy\right) |u|^{2^{\ast}_{\mu}-2}u\ \text{in}\ \Omega\right\rbrace
	\end{equation*}
	and
	\begin{equation}\label{S}
		S_{\mu}^{+}:=S_{\mu}\cap\left\lbrace u\in H^{1}_{0}(\Omega)\ |\ u(x)\geq0\ \text{a.e. in}\ \Omega\right\rbrace .
	\end{equation}
We have the following properties about the stable set $W_{\mu}$ and the unstable set $V_{\mu}$.
	
	\begin{Prop}\label{P-Properties}
		Let $W_{\mu}, V_{\mu}$ and $S_{\mu}^{+}$ as given in (\ref{W})-(\ref{V}) and (\ref{S}) respectively. Then
		\begin{item}
			$(i)$ $W_{\mu}$ is a bounded neighbourhood of $0$ in $H^{1}_{0}(\Omega)$ and is star-shaped with respect to $0$,\\
			$(ii)$ $V_{\mu}$ is arcwise connected in $H^{1}_{0}(\Omega)$,\\
			$(iii)$ if $\Omega$ is a star-shaped domain, then $S_{\mu}^{+}=\{0\}$,\\
			$(iv)$ $\overline{W_{\mu}}=\left\{u\in H^{1}_{0}(\Omega)\ |\ J_{\mu}(u)\leq m_{\mu},\ I_{\mu}(u)>0\right\}\cup\left\lbrace 0\right\rbrace$,\\
			$(v)$ $0\notin\overline{V_{\mu}}=\left\lbrace u\in H^{1}_{0}(\Omega)\ |\ J_{\mu}(u)\leq m_{\mu},\ I_{\mu}(u)<0\right\rbrace$,\\
			$(vi)$ $\overline{W_{\mu}}\cap\overline{V_{\mu}}=\emptyset$,\\
		\end{item}
		where $\overline{W_{\mu}}$ and $\overline{V_{\mu}}$ denote the closure of $W_{\mu}$ and $V_{\mu}$ in $H^{1}_{0}(\Omega)$ respectively.
	\end{Prop}
	\begin{proof}
		$(i)$ Obviously $0\in W_\mu$. Assume that $u\in W_{\mu}$ and $u\not\equiv0$. Then we have
		\begin{equation*}
			A(u)-B(u)>0.
		\end{equation*}
		Furthermore, we derive
		\begin{equation*}
			m_{\mu}>J_{\mu}(u)=\frac{1}{2}A(u)
			-\frac{1}{22^{\ast}_{\mu}}B(u)
			\geq\frac{2^{\ast}_{\mu}-1}{22^{\ast}_{\mu}}A(u),
		\end{equation*}
		which implies that
		\begin{equation*}
			A(u)\leq\frac{2(2N-\mu)}{N-\mu+2}m_{\mu},
		\end{equation*}
		from which we conclude that $W_\mu$ is bounded. In the other hand, let $u\in H^1_0(\Omega)$ with $\|u\|_{H^1_0(\Omega)}<<1$. Then from Lemma \ref{L-HLS}, it is easy to get $u\in W_\mu$. Furthermore, if $u\in W_\mu$, then $I_\mu(tu)>0$ and $tu\in W_\mu$ for $t\in [0,1]$. Therefore, $(i)$ holds.

		$(ii)$ Let $u_1,u_2\in V_\mu$ be arbitrary but fixed. By $(i)$, we know that there exists $r>0$ such that
		\begin{equation*}
			W_\mu\subset B_{r}(0):=\{u\in H^{1}_{0}(\Omega)\ |\ A(u)<r\}
		\end{equation*}
		and there exist $\alpha_1,\alpha_2\in [1,\infty)$ such that $\alpha_1u_1,\alpha_2u_2\in B^{c}_{r}(0)$, where
		\begin{equation*}
			B^{c}_{r}(0):=\left\lbrace u\in H^{1}_{0}(\Omega)\ |\ J_{\mu}(u)<m_{\mu}\right\rbrace -B_{r}(0).
		\end{equation*}
		Since $B^{c}_{r}(0)$ is a arcwise connected, we can connect $\alpha_1u_1$ with $\alpha_2u_2$ by path in $V_\mu$. Obviously, we can also connect $\alpha_1u_1$ with $u_1$ ($\alpha_2u_2$ and $u_2$ respectively) by paths in $V_\mu$ respectively, which implies that $V_\mu$ is aecwise connected.
		
		$(iii)$ By the classical Pohozaev identity and taking into account that $\Omega$ is star-shaped, we complete easily the proof of $(iii)$.
		
		$(iv)$ If $u\in \overline{W_{\mu}}\backslash\left\lbrace 0\right\rbrace $, then one has
		$J_{\mu}(u)\leq m_{\mu}$ and $I_{\mu}(u)\geq0$. It is easy to see that $J_{\mu}(u)\geq m_{\mu}$  provided $I_{\mu}(u)=0$. Then, we derive that $J_{\mu}(u)<m_{\mu}$ and $I_{\mu}(u)=0$ do not happen simultaneously.
		Next, since $S_{H,L}$ is not achieved when $\Omega\not=\mathbb{R}^{N}$, the case where $J_{\mu}(u)=m_{\mu}$ and $I_{\mu}(u)=0$ does not occur. Therefore, $\overline{W_{\mu}}\subset\left\{u\in H^{1}_{0}(\Omega)\ |\ J_{\mu}(u)\leq m_{\mu},\ I_{\mu}(u)>0\right\}\cup\left\lbrace 0\right\rbrace$. In the other hand, let $u\in H^1_0(\Omega)$ such that $J_{\mu}(u)=m_{\mu},\ I_{\mu}(u)>0$. Then , it is easy to see that for any $t<1$ and close to $1$, we have that $J_{\mu}(tu)<m_{\mu},\ I_{\mu}(tu)>0$. Taking a sequence $(t_n)_{n\in \mathbb N} $ such that $t_n\to 1^-$ and $u_n=t_nu$, we get $u\in \overline{W_{\mu}}$ and $(iv)$ holds.
		
		$(v)$ Similar to the argument of $(iv)$, we have
		\begin{equation*}
			\overline{V_{\mu}}=\left\{u\in H^{1}_{0}(\Omega)\ |\ J_{\mu}(u)\leq m_{\mu},\ I_{\mu}(u)<0\right\}.
		\end{equation*}
		Next, we prove that $0\not\in\overline{V_{\mu}}$. On the contrary, that is $0\in\overline{V_{\mu}}$. Then there exists a sequence $\left\lbrace u_{n}\right\rbrace\subset V_{\mu}$ such that $u_{n}\rightarrow0$ as $n\rightarrow\infty$ in $H^{1}_{0}(\Omega)$. It follows from  $(i)$ that $u_{n}\in W_{\mu}$ for $n$ sufficiently large, which contradicts the fact that $W_{\mu}\cap V_{\mu}=\emptyset$. Therefore,  statement$(v)$ is valid.
		
		$(vi)$ The argument is a direct consequence of $(iv)$ and $(v)$.\newline
		Consequently, the proof of Proposition \ref{P-Properties} is now completed.
	\end{proof}
	
	\begin{Rem}\label{R2}
		Suppose that $\Omega$ is a bounded domain in $\mathbb{R}^{N}$. Set
		\begin{equation*}
			W_{1}:=\left\{u\in H^{1}_{0}(\Omega)\ |\ J_{\mu}(u)<m_{\mu},\ B(u)<S^{\frac{2N-\mu}{N-\mu+2}}_{H,L}\right\}
		\end{equation*}
		and
		\begin{equation*}
			V_{1}:=\left\{u\in H^{1}_{0}(\Omega)\ |\ J_{\mu}(u)<m_{\mu},\ B(u)>S^{\frac{2N-\mu}{N-\mu+2}}_{H,L}\right\}.
		\end{equation*}
		Then $W_{\mu}=W_{1}$ and $V_{\mu}=V_{1}$.
	\end{Rem}
	
	Indeed, assume that $J_{\mu}(u)\leq m_{\mu}$ and
	\begin{equation}\label{a1}
		B(u)=S^{\frac{2N-\mu}{N-\mu+2}}_{H,L},
	\end{equation}
	then we derive that
	\begin{equation}\label{a2}
		A(u)\leq S^{\frac{2N-\mu}{N-\mu+2}}_{H,L}.
	\end{equation}
	Furthermore, it follows from (\ref{SHL}) and (\ref{a1})-(\ref{a2}) that
	\begin{equation*}
		S^{\frac{2N-\mu}{N-\mu+2}}_{H,L}
		=S_{H,L}\left(B(u)\right) ^{\frac{N-2}{2N-\mu}}
		\leq A(u)
		\leq S^{\frac{2N-\mu}{N-\mu+2}}_{H,L},
	\end{equation*}
	which implies that
	$A(u)=S_{H,L}\left(B(u)\right) ^{\frac{N-2}{2N-\mu}}$. Hence, $u=0$ provided $\Omega\not=\mathbb{R}^{N}$. Therefore, 
	\begin{equation}\label{a3}
		\left\{u\in H^{1}_{0}(\Omega)\ |\ J_{\mu}(u)<m_{\mu}\right\}
		=W_{1}\cup V_{1},
	\end{equation}
	and
	\begin{equation}\label{a4}
		W_{1}\cap V_{1}=\emptyset,
	\end{equation}
	from which we easily get that 
	\begin{equation}\label{W-V}
	W_{1}\cup V_{1}=W_\mu \cup V_\mu 
	\end{equation}
and $W_{1}$, $V_{1}$ are open.
	By the definition of $V_{1}$, for any $u\in V_{1}$, one has
	\begin{equation*}
		A(u)
		<\frac{N-\mu+2}{2N-\mu}S^{\frac{2N-\mu}{N-\mu+2}}_{H,L}+\frac{N-2}{2N-\mu}B(u)<B(u),
	\end{equation*}
	which implies that $u\in V_{\mu}$. Therefore, $V_{1}\subset V_{\mu}$.
	
	Furthermore, since $V_{\mu}$ is open connected
	and combining with \eqref{W-V}
	\begin{equation}\label{a5}
		V_{\mu}=V_{1}\cup(W_{1}\cap V_{\mu}).
	\end{equation}
	By the definition of $W_{1}$ and (\ref{SHL}), for any $u\in W_{1}$, one has
	\begin{equation*}
		B(u)
		\leq(B(u))^{\frac{N-2}{2N-\mu}}S_{H,L}\leq A(u),
	\end{equation*}
	which implies that $W_{1}\cap V_{\mu}=\emptyset$. Therefore, we derive that $V_{\mu}=V_{1}$.
	
	On the other hand, by \eqref{W-V}, it is clear that
	\begin{equation}\label{a6}
		\left\{u\in H^{1}_{0}(\Omega)\ |\ J_{\mu}(u)<m_{\mu}\right\}
		=W_{\mu}\cup V_{\mu}\ \text{and}\ W_{\mu}\cap V_{\mu}=\emptyset.
	\end{equation}
	Then it follows from (\ref{a3})-(\ref{a6})  that $W_{\mu}=W_{1}$. The proof is now complete.$\hfill{} \Box$
	
	Next, we shall show the local well-posedness of problem \eqref{E} in $H^{1}_{0}(\Omega)$. Denote $A:=-\Delta$ with Dirichlet boundary conditions. Following \cite{Tanabe} (or see \cite{Hoshino-Yamada}), we define the fractional powers $A^{\gamma}$ with domain $D(A^{\gamma})$ of operator $A$ for all $\gamma\in(0,1)$ and the analytic semigroup of bounded linear operators $\left\lbrace e^{-tA}\right\rbrace _{t\geq0}$ generated by $-A$. Let
	\begin{equation}\label{Integral Equation}
		\Phi_{u_{0}}[u]:=e^{-tA}u_{0}+\int_{0}^{t}e^{-(t-\tau)A}\left\lbrace \left(\int_{\Omega}\frac{|u(y,\tau)|^{p}}{|x-y|^{\mu}}dy\right)|u(\tau)|^{2^{\ast}_{\mu}-2}u(\tau)\right\rbrace d\tau.
	\end{equation}
	And for any $K>0$, $\alpha,\, \beta\in (0,1)$ we set
	\begin{equation}\label{Space}
		\begin{aligned}
			Y_{T,K}:=\bigg\lbrace&u\in BC\Large((0,T];D(A^{\alpha})\cap D(A^{\beta})\Large), \\
			&\max\bigg\lbrace \sup_{t\in (0,T]}t^{\alpha-\frac{1}{2}}\Vert A^{\alpha}u(t)\Vert_{2}, \sup_{t\in (0,T]}t^{\beta-\frac{1}{2}}\Vert A^{\beta}u(t)\Vert_{2}\bigg\rbrace\leq K\bigg\rbrace
		\end{aligned}
	\end{equation}
	$BC((0,T];\, D(A^\alpha)\cap D(A^\beta))$ denotes the set of functions bounded and continuous from $(0,T]$ in $D(A^\alpha)\cap D(A^\beta)$. Now, to get existence of solutions to \eqref{E}, our task is reduced to applying the Fixed Point Theorem to the mapping $\Phi_{u_{0}}[u]$. More precisely, we prove that $\Phi_{u_{0}}[u]$ is a contraction mapping in $Y_{T,K}$ for a suitable $K>0$. For this purpose,
	we first recall the following classical properties of Dirichlet heat semigroup and give some estimates. In the sequel, $\lambda$ denotes any fixed real number such that $0<\lambda<\lambda_1(\Omega):=\inf\{\|\nabla v\|_{L^2(\Omega)},\; v\in H^1_0(\Omega),\;\|v\|_{L^2(\Omega)}=1\}$.
	
	\begin{lem}\label{L-Heat semigroup properties1}
		{\rm \cite[Lemma 2.1]{Hoshino-Yamada}} Suppose that $X$ be a Banach space with norm $\Vert \cdot\Vert_{X}$. For each $\gamma\geq0$, there exists a positive constant $C(\gamma)>0$ such that
		\begin{equation*}
			\Vert A^{\gamma}e^{-tA}w\Vert_{X}
			\leq C(\gamma)t^{-\gamma}e^{-\lambda t}\Vert w\Vert_{X}
		\end{equation*}
		for all $w\in X$ and $t>0$.
		Moreover, for each $\gamma>0$, there holds
		\begin{equation*}
			t^{\gamma}\Vert A^{\gamma}e^{-tA}w\Vert_{X}\rightarrow0\ \text{as}\ t\rightarrow 0^+
		\end{equation*}
		for all $w\in X$.
	\end{lem}
	\begin{lem}\label{L-Heat semigroup properties2}
		{\rm \cite[Proposition 48.4$^{\ast}$]{Quittner-Souplet}}
		Let $N\geq 1$. Then for any $1\leq r_{1}<r_{2}\leq \infty$, there holds
		\begin{equation*}
			\Vert e^{t\Delta}w\Vert_{r_{2}}
			\leq (4\pi t)^{-\frac{N}{2}(\frac{1}{r_{1}}-\frac{1}{r_{2}})}\Vert w\Vert_{r_{1}}
		\end{equation*}
		for all $t>0$ and $w\in L^{r_{1}}(\Omega)$.
	\end{lem}
	
	Letting $\theta>0$ be an arbitrary and fixed constant, we define:
	\begin{equation}\label{Index}
		\alpha:=\frac{1}{2}+\frac{N-2}{2^{\theta}}\ \text{and}\
		\beta:=\frac{2^{\theta-1}(2N-\mu)-(N-2)(2N-\mu)}{2^{\theta}(N-\mu+2)}.
	\end{equation}
	Choosing $\theta>0$ large enough such that $\alpha>\frac{1}{2}$ and close to $\frac{1}{2}$. Clearly, $\alpha,\beta\in (\frac{1}{2},1)$ provided $N>2$ and $0<\mu<\min\left\lbrace N,4\right\rbrace $.
	Set
	\begin{equation*}
		f(u(t)):=\left(\int_{\Omega}\frac{|u(y,t)|^{2^{\ast}_{\mu}}}{|x-y|^{\mu}}dy\right)|u(t)|^{2^{\ast}_{\mu}-2}u(t).
	\end{equation*}
	Then, we have the following estimates for the nonlinear term $f(u(t))$.
	
	\begin{lem}\label{L-Nonlinear estimate}
		Assume that $0<\mu<\min\left\lbrace N,4\right\rbrace $ and $\alpha,\beta$ as in (\ref{Index}). Then for any $u,v\in Y_{T,K}$ , we have
		\begin{equation}\label{nonlinear-estimate1}
			\left\Vert \int_{0}^{t}A^{\gamma}e^{-(t-\tau)A}f(u(\tau))d\tau\right\Vert_{2}\leq Ct^{1-\gamma-(2^{\ast}_{\mu}-1)(\beta-\frac{1}{2})-2^{\ast}_{\mu}(\alpha-\frac{1}{2})}K^{22^{\ast}_{\mu}-1}
		\end{equation}
		and
		\begin{equation}\label{nonlinear-estimate2}
			\begin{aligned}
				\bigg\Vert \int_{0}^{t}A^{\gamma}e^{-(t-\tau)A}&(f(u(\tau))-f(v(\tau)))d\tau\bigg\Vert_{2}\\
				 \leq&Ct^{1-\gamma-(2^{\ast}_{\mu}-1)(\beta-\frac{1}{2})-2^{\ast}_{\mu}(\alpha-\frac{1}{2})}(2K)^{2(2^{\ast}_{\mu}-1)}\Vert\vert u-v\Vert\vert_{Y_{K,T}}
			\end{aligned}
		\end{equation}
		for some $C>0$.
	\end{lem}
	\begin{proof}
		First, we prove \eqref{nonlinear-estimate1}. From the Hardy-Littlewood-Sobolev inequality (see Lemma \ref{L-HLS}) and a duality argument, we have that 
			\begin{equation}\label{dual HLS}
				\left\Vert \int_{\Omega}\frac{\bar{f}(y)}{|x-y|^{\mu}}dy\right\Vert_{s}\leq C\Vert \bar{f}\Vert_{s_{1}}
			\end{equation}
		where $\frac{1}{s_{1}}=\frac{N-\mu}{N}+\frac{1}{s}$.
			Then, from the H\"older inequality, it follows that
		\begin{equation}\label{b1}
			\begin{aligned}
				\Vert f(&u(t))\Vert^{2}_{2}=\int_{\Omega}\left\vert \left(\int_{\Omega}\frac{|u(y,t)|^{2^{\ast}_{\mu}}}{|x-y|^{\mu}}dy\right)|u(t)|^{2^{\ast}_{\mu}-2}u(t)\right\vert^{2}dx\\
				\leq&\left\Vert \int_{\Omega}\frac{|u(y,t)|^{2^{\ast}_{\mu}}}{|x-y|^{\mu}}dy\right\Vert^{2}_{2q}\left( \int_{\Omega}\left\vert |u(t)|^{2^{\ast}_{\mu}-2}u(t)\right\vert^{2q^{\prime}}dx\right) ^{1/q^{\prime}}
				\leq C\Vert u(t)\Vert^{22^{\ast}_{\mu}}_{2^{\ast}_{\mu}r}
				\Vert u(t)\Vert^{2(2^{\ast}_{\mu}-1)}_{2(2^{\ast}_{\mu}-1)q^{\prime}},
			\end{aligned}
		\end{equation}
		where $q, q^{\prime}\in(1,+\infty)$ are conjugate and $r\in(1,+\infty)$ satisfies
		\begin{equation}\label{b2}
			\frac{1}{r}=\frac{1}{2q}+\frac{N-\mu}{N}.
		\end{equation}
		Taking
		\begin{equation}\label{q-r}
			q=\frac{2^{\theta-2}N}{2^{\theta-2}\mu-(2N-\mu)},\ q^{\prime}=\frac{2^{\theta-2}N}{2^{\theta-2}(N-\mu)+(2N-\mu)},\ r=\frac{2N}{2N-\mu}\frac{2^{\theta-2}}{2^{\theta-2}-1},
		\end{equation}
		obviously, $q, q^{\prime}\in(1,+\infty)$ are conjugate and $r\in(1,+\infty)$ satisfies (\ref{b2}).
		From (2) in \cite[Proposition 4.1]{Badra-Bal-Giacomoni} (see also  \cite{Henry} for further properties of fractional powers of sectorial operators) and since
		\begin{equation*}
			2^{\ast}_{\mu}r=\frac{2N}{N-4\alpha}\ \text{and}\ 2(2^{\ast}_{\mu}-1)q^{\prime}=\frac{2N}{N-4\beta},
		\end{equation*}
		together with the Sobolev Imbedding Theorem (see \cite{Henry}), one has
		\begin{equation}\label{b3}
			\Vert u(t)\Vert_{2^{\ast}_{\mu}r}\leq C\Vert A^{\alpha}u(t)\Vert_{2},\ \text{for all}\ u(t)\in D(A^{\alpha})
		\end{equation}
		and
		\begin{equation}\label{b4}
			\Vert u(t)\Vert_{2(2^{\ast}_{\mu}-1)q^{\prime}}\leq C\Vert A^{\beta}u(t)\Vert_{2},\ \text{for all}\ u(t)\in D(A^{\beta}).
		\end{equation}
		Since for $\theta>1$ large, $2^*_{\mu}(\alpha-\frac{1}{2})$ and $(2^*_{\mu}-1)(\beta-\frac{1}{2})$ are respectively close to $0$ and $\frac{1}{2}$,  it follows from Lemma \ref{L-Heat semigroup properties1}, (\ref{b1}) and (\ref{b3})-(\ref{b4}) that
		\begin{equation}\nonumber
			\begin{aligned}
				\left\Vert \int_{0}^{t}A^{\gamma}e^{-(t-\tau)A}f(u(\tau))d\tau\right\Vert_{2}\leq& C\int_{0}^{t}(t-\tau)^{-\gamma}e^{-\lambda(t-\tau)}\Vert f(u(\tau))\Vert_{2}d\tau\\
				\leq&C\int_{0}^{t}(t-\tau)^{-\gamma}\Vert A^{\alpha}u(\tau)\Vert^{2^{\ast}_{\mu}}_{2}\Vert A^{\beta}u(\tau)\Vert^{2^{\ast}_{\mu}-1}_{2}d\tau \\
				 \leq&C\int_{0}^{t}(t-\tau)^{-\gamma}\tau^{-(2^{\ast}_{\mu}-1)(\beta-\frac{1}{2})-2^{\ast}_{\mu}(\alpha-\frac{1}{2})}\Vert\vert u\Vert\vert_{Y_{T,K}}^{22^{\ast}_{\mu}-1}d\tau\\
				 \leq&Ct^{1-\gamma-(2^{\ast}_{\mu}-1)(\beta-\frac{1}{2})-2^{\ast}_{\mu}(\alpha-\frac{1}{2})}K^{22^{\ast}_{\mu}-1}
			\end{aligned}
		\end{equation}
		for $u\in Y_{T,K}$, where $C's$ are different constants from line to line.
		Consequently, the proof of (\ref{nonlinear-estimate1}) is complete.

		Next, we prove (\ref{nonlinear-estimate2}).
		Note that
		\begin{equation}\label{b11}
			\begin{aligned}
				\Vert f(u(t))-&f(v(t))\Vert^{2}_{2}
				\leq\int_{\Omega}\bigg\vert \left(\int_{\Omega}\frac{|u(y,t)|^{2^{\ast}_{\mu}}}{|x-y|^{\mu}}dy\right)\left( |u(t)|^{2^{\ast}_{\mu}-2}u(t)-|v(t)|^{2^{\ast}_{\mu}-2}(t)v\right)\\
				&+\left(\int_{\Omega}\frac{|u(y,t)|^{2^{\ast}_{\mu}}-|v(y,t)|^{2^{\ast}_{\mu}}}{|x-y|^{\mu}}dy \right)|v(t)|^{2^{\ast}_{\mu}-2}v(t)\bigg\vert^{2}dx\\
				&\leq2\int_{\Omega}\left\vert \left(\int_{\Omega}\frac{|u(y,t)|^{2^{\ast}_{\mu}}}{|x-y|^{\mu}}dy\right)\left( |u(t)|^{2^{\ast}_{\mu}-2}u(t)-|v(t)|^{2^{\ast}_{\mu}-2}v(t)\right)\right\vert^{2}dx\\
				&+2\int_{\Omega}\left\vert\left(\int_{\Omega}\frac{|u(y,t)|^{2^{\ast}_{\mu}}-|v(y,t)|^{2^{\ast}_{\mu}}}{|x-y|^{\mu}}dy \right)|v(t)|^{2^{\ast}_{\mu}-2}v(t)\right\vert^{2}dx
				=:2(I_{1}+I_{2}).
			\end{aligned}
		\end{equation}
		Similar to the proof of (\ref{b1}), it follows from the H\"older inequality, the Hardy-Littlewood-Sobolev inequality and the Mean Value Theorem that
		\begin{equation}\label{b12}
			\begin{aligned}
				I_{1}
				\leq&C\left\Vert \int_{\Omega}\frac{|u(y,t)|^{2^{\ast}_{\mu}}}{|x-y|^{\mu}}dy\right\Vert^{2}_{2q}\left( \int_{\Omega}\left\vert |\xi|^{2^{\ast}_{\mu}-2}(u(t)-v(t))\right\vert^{2q^{\prime}}dx\right) ^{1/q^{\prime}}\\
				\leq&C\Vert u(t)\Vert^{22^{\ast}_{\mu}}_{2^{\ast}_{\mu}r}\left( \int_{\Omega}|\xi|^{2(2^{\ast}_{\mu}-2)q^{\prime}\frac{2^{\ast}_{\mu}-1}{2^{\ast}_{\mu}-2}}dx\right) ^{\frac{2^{\ast}_{\mu}-2}{{(2^{\ast}_{\mu}-1)q^{\prime}}}}
				\left( \int_{\Omega}\vert u(t)-v(t)\vert ^{2q^{\prime}(2^{\ast}_{\mu}-1)}dx\right) ^{\frac{1}{(2^{\ast}_{\mu}-1)q^{\prime}}}\\
				\leq&C\Vert u(t)\Vert^{22^{\ast}_{\mu}}_{2^{\ast}_{\mu}r}
				\left( \Vert u(t)\Vert_{2(2^{\ast}_{\mu}-1)q^{\prime}}+\Vert v(t)\Vert_{2(2^{\ast}_{\mu}-1)q^{\prime}}\right)  ^{2(2^{\ast}_{\mu}-2)}\Vert u(t)-v(t)\Vert_{2(2^{\ast}_{\mu}-1)q^{\prime}} ^{2},
			\end{aligned}
		\end{equation}
		where $\xi=\xi(x,t)$ is a real real number between $u(x,t)$ and $v(x,t)$ and $q, q^{\prime},r$ are given in (\ref{q-r})
		and $C's$ are different constants from one line to another.
		Similarly, we also have
		\begin{equation}\label{b14}
			\begin{aligned}
				I_{2}
				\leq&\left( \int_{\Omega}\left\vert\int_{\Omega}\frac{|u(y,t)|^{2^{\ast}_{\mu}}-|v(y,t)|^{2^{\ast}_{\mu}}}{|x-y|^{\mu}}dy \right\vert^{2q}dx\right)^{1/q}
				\left( \int_{\Omega}\left\vert|v(t)|^{2^{\ast}_{\mu}-2}v(t) \right\vert^{2q\prime}dx\right)^{1/q^{\prime}}\\
				\leq&C\Vert v(t)\Vert^{2(2^{\ast}_{\mu}-1)}_{2(2^{\ast}_{\mu}-1)q^{\prime}}\left(\Vert u(t)\Vert_{2^{\ast}_{\mu}r}+\Vert v(t)\Vert_{2^{\ast}_{\mu}r}\right) ^{2(2^{\ast}_{\mu}-1)}\Vert u(t)-v(t)\Vert^{2}_{2^{\ast}_{\mu}r}.
			\end{aligned}
		\end{equation}
		Hence, it follows from (\ref{b11})-(\ref{b14}) that
		\begin{equation}\label{b15}
			\begin{aligned}
				\Vert f(u(t))-&f(v(t))\Vert_{2}\\
				\leq&C\Vert u(t)\Vert^{2^{\ast}_{\mu}}_{2^{\ast}_{\mu}r}
				\left( \Vert u(t)\Vert_{2(2^{\ast}_{\mu}-1)q^{\prime}}+\Vert v(t)\Vert_{2(2^{\ast}_{\mu}-1)q^{\prime}}\right)  ^{2^{\ast}_{\mu}-2}\Vert u(t)-v(t)\Vert_{2(2^{\ast}_{\mu}-1)q^{\prime}}\\
				&+C\Vert v(t)\Vert^{2^{\ast}_{\mu}-1}_{2(2^{\ast}_{\mu}-1)q^{\prime}}\left(\Vert u(t)\Vert_{2^{\ast}_{\mu}r}+\Vert v(t)\Vert_{2^{\ast}_{\mu}r}\right) ^{2^{\ast}_{\mu}-1}\Vert u(t)-v(t)\Vert_{2^{\ast}_{\mu}r}.
			\end{aligned}
		\end{equation}
		Furthermore, by  (\ref{b3})-(\ref{b4}) and (\ref{b15}), we derive that
		\begin{equation}\label{b16}
			\begin{aligned}
				\Vert f(u(t))-&f(v(t))\Vert_{2}\\
				\leq&C\Vert A^{\alpha}u(t)\Vert^{2^{\ast}_{\mu}}_{2}
				\left( \Vert A^{\beta}u(t)\Vert_{2}+\Vert A^{\beta}v(t)\Vert_{2}\right)  ^{2^{\ast}_{\mu}-2}\Vert A^{\beta}(u(t)-v(t))\Vert_{2}\\
				&+C\Vert A^{\beta}v(t)\Vert^{2^{\ast}_{\mu}-1}_{2}\left(\Vert A^{\alpha}u(t)\Vert_{2}+\Vert A^{\alpha}v(t)\Vert_{2}\right) ^{2^{\ast}_{\mu}-1}\Vert A^{\alpha}(u(t)-v(t))\Vert_{2}.
			\end{aligned}
		\end{equation}
		Therefore, by Lemma \ref{L-Heat semigroup properties1} and (\ref{b16}), we obtain
		\begin{equation}\nonumber
			\begin{aligned}
				\bigg\Vert\int_{0}^{t}A^{\gamma}&e^{-(t-\tau)A}(  f(u(\tau))-f(v(\tau)))d\tau\bigg\Vert_{2}
				\leq C\int_{0}^{t}(t-\tau)^{-\gamma}e^{-\lambda (t-\tau)}\Vert f(u(\tau))-f(v(\tau))\Vert_{2}d\tau\\
				\leq&C\int_{0}^{t}(t-\tau)^{-\gamma}\bigg[\Vert A^{\alpha}u(\tau)\Vert^{2^{\ast}_{\mu}}_{2}
				\left(\Vert A^{\beta}u(\tau)\Vert_{2}+\Vert A^{\beta}v(\tau)\Vert_{2}\right)^{2^{\ast}_{\mu}-2}\Vert A^{\beta}(u(\tau)-v(\tau))\Vert_{2}\\
				&\ \ \ \ +\Vert A^{\beta}v(\tau)\Vert^{2^{\ast}_{\mu}-1}_{2}
				\left(\Vert A^{\alpha}u(\tau)\Vert_{2}+\Vert A^{\alpha}v(\tau)\Vert_{2}\right) ^{2^{\ast}_{\mu}-1}\Vert A^{\alpha}(u(\tau)-v(\tau))\Vert_{2}\bigg] d\tau\\
				 \leq&Ct^{1-\gamma-(2^{\ast}_{\mu}-1)(\beta-\frac{1}{2})-2^{\ast}_{\mu}(\alpha-\frac{1}{2})}(2K)^{2(2^{\ast}_{\mu}-1)}\Vert\vert u-v\Vert\vert_{Y_{T,K}}
			\end{aligned}
		\end{equation}
		for $u,v\in Y_{T,K}$, where $C's$ are different constants from line to line. 
	\end{proof}
	We now prove an existence result from which follows directly the proof of Theorem \ref{thm-existence}:
	\begin{Prop}\label{P-Local existence}
			Suppose that $0<\mu<\min\left\lbrace N,4\right\rbrace $ and $\alpha, \beta$ as in \eqref{Index}. Then for each $u_{0}\in H^{1}_{0}(\Omega)$, there exists $T=T(\Vert u_{0}\Vert_{H^{1}_{0}})>0$ such that \eqref{PI} has a unique solution $u\in C([0,T];H^{1}_{0}(\Omega))$.
			Furthermore,  $u$ is a classical solution of problem \eqref{E} in $(0,T]$. Moreover, \eqref{bbb5} and \eqref{bbb6} hold.
	\end{Prop}
	
	\begin{proof}
		We study problem \eqref{PI} by a contraction mapping argument. For any $u\in Y_{T,K}$, by (\ref{Integral Equation}) and (\ref{nonlinear-estimate1}), we notice that
		\begin{equation}\label{c1}
			\begin{aligned}
				\sup_{t\in (0,T]}t^{\gamma-\frac{1}{2}}\Vert A^{\gamma}\Phi_{u_{0}} [u](t)\Vert_{2}
				\leq&\sup_{t\in (0,T]}t^{\gamma-\frac{1}{2}}\left\lbrace  \Vert A^{\gamma} e^{-tA}u_{0}\Vert_{2}+\left\Vert \int_{0}^{t}A^{\gamma}e^{-(t-\tau)A}f(u(\tau))d\tau\right\Vert_{2}\right\rbrace \\
				\leq&\sup_{t\in (0,T]}\left\lbrace t^{\gamma-\frac{1}{2}}\Vert A^{\gamma} e^{-tA}u_{0}\Vert_{2}
				 +Ct^{\frac{1}{2}-(2^{\ast}_{\mu}-1)(\beta-\frac{1}{2})-2^{\ast}_{\mu}(\alpha-\frac{1}{2})}K^{22^{\ast}_{\mu}-1}\right\rbrace .
			\end{aligned}
		\end{equation}
		Taking $\gamma=\alpha$ and $\gamma=\beta$ in (\ref{c1}) respectively, one has
		\begin{equation}\label{c2}
			\begin{aligned}
				\Vert\vert \Phi_{u_{0}} [u]\Vert\vert_{Y_{T,K}}\leq&\max\left\lbrace\sup_{t\in (0,T]}\left\lbrace t^{\alpha-\frac{1}{2}}\Vert A^{\alpha} e^{-tA}u_{0}\Vert_{2}
				 +Ct^{\frac{1}{2}-(2^{\ast}_{\mu}-1)(\beta-\frac{1}{2})-2^{\ast}_{\mu}(\alpha-\frac{1}{2})}K^{22^{\ast}_{\mu}-1}\right\rbrace ,\right.\\
				&\left. \ \ \ \ \ \ \ \ \sup_{t\in (0,T]}\left\lbrace t^{\beta-\frac{1}{2}}\Vert A^{\beta} e^{-tA}u_{0}\Vert_{2}
				 +Ct^{\frac{1}{2}-(2^{\ast}_{\mu}-1)(\beta-\frac{1}{2})-2^{\ast}_{\mu}(\alpha-\frac{1}{2})}K^{22^{\ast}_{\mu}-1}\right\rbrace \right\rbrace\\
				\leq&\max\left\lbrace \sup_{t\in (0,T]}t^{\alpha-\frac{1}{2}}\Vert A^{\alpha} e^{-tA}u_{0}\Vert_{2},\sup_{t\in (0,T]}t^{\beta-\frac{1}{2}}\Vert A^{\beta} e^{-tA}u_{0}\Vert_{2}\right\rbrace\\
				&+C\sup_{t\in (0,T]}t^{\frac{1}{2}-(2^{\ast}_{\mu}-1)(\beta-\frac{1}{2})-2^{\ast}_{\mu}(\alpha-\frac{1}{2})}K^{22^{\ast}_{\mu}-1},
			\end{aligned}
		\end{equation}
		where $C's$ are different constants from line to line.
		By (\ref{Index}), we note that
		\begin{equation}\label{key}
			\frac{1}{2}-(2^{\ast}_{\mu}-1)(\beta-\frac{1}{2})-2^{\ast}_{\mu}(\alpha-\frac{1}{2})=0.
		\end{equation}
		Hence, by (\ref{c2}) and \eqref{key}, if $T,K$ satisfy
		\begin{equation}\label{kt1}
			CK^{22^{\ast}_{\mu}-1}<K
		\end{equation}
		and
		\begin{equation}\label{kt2}
			\max\left\lbrace \sup_{t\in (0,T]}t^{\alpha-\frac{1}{2}}\Vert A^{\alpha} e^{-tA}u_{0}\Vert_{2},\sup_{t\in (0,T]}t^{\beta-\frac{1}{2}}\Vert A^{\beta} e^{-tA}u_{0}\Vert_{2}\right\rbrace \leq K-CK^{22^{\ast}_{\mu}-1},
		\end{equation}
		then we derive $\Vert\vert \Phi [u]\Vert\vert_{Y_{T,K}}\leq K$.
		Therefore, $\Phi$ maps $Y_{T,K}$ into itself.
		
		Next, the same procedure enables us to prove that $\Phi_{u_{0}}$ is a contraction mapping.
		For given $u, v\in Y_{T,K}$,
		it follows from (\ref{Integral Equation}) and (\ref{nonlinear-estimate2}) that
		\begin{equation}\label{c3}
			\begin{aligned}
				\sup_{t\in (0,T]}t^{\gamma-\frac{1}{2}}\Vert A^{\gamma}\Phi_{u_{0}} [u](t)&-A^{\gamma}\Phi_{u_{0}} [v](t)\Vert_{2}\\
				\leq& \sup_{t\in (0,T]}t^{\gamma-\frac{1}{2}}\left\Vert \int_{0}^{t}A^{\gamma}e^{-(t-\tau)A}(f(u(\tau))-f(v(\tau)))d\tau\right\Vert_{2}\\
				\leq&\sup_{t\in (0,T]}Ct^{\frac{1}{2}-(2^{\ast}_{\mu}-1)(\beta-\frac{1}{2})-2^{\ast}_{\mu}(\alpha-\frac{1}{2})}(2K)^{2(2^{\ast}_{\mu}-1)}\Vert\vert u-v\Vert\vert_{Y_{T,K}}.
			\end{aligned}
		\end{equation}
		Similarly, taking $\gamma=\alpha$ and $\gamma=\beta$ in (\ref{c3}) respectively and by (\ref{key}), one obtain that
		\begin{equation}\label{c4}
			\begin{aligned}
				\Vert\vert \Phi_{u_{0}} [u]-\Phi_{u_{0}} [v]\Vert\vert_{Y_{T,K}}
				\leq& C\sup_{t\in (0,T]}t^{\frac{1}{2}-(2^{\ast}_{\mu}-1)(\beta-\frac{1}{2})-2^{\ast}_{\mu}(\alpha-\frac{1}{2})}(2K)^{2(2^{\ast}_{\mu}-1)}\Vert\vert u-v\Vert\vert_{Y_{T,K}}\\
				\leq&C(2K)^{2(2^{\ast}_{\mu}-1)}\Vert\vert u-v\Vert\vert_{Y_{T,K}}.
			\end{aligned}
		\end{equation}
		Therefore, if $T,K$ satisfy
		\begin{equation}\label{kt3}
			C(2K)^{2(2^{\ast}_{\mu}-2)}\leq\frac{1}{2},
		\end{equation}
		then
		\begin{equation*}
			\Vert\vert \Phi [u]-\Phi [v]\Vert\vert_{Y_{T,K}}\leq
			C(2K)^{2(2^{\ast}_{\mu}-2)}\Vert\vert u-v\Vert\vert_{Y_{T,K}}\leq\frac{1}{2}\Vert\vert u-v\Vert\vert_{Y_{T,K}}.
		\end{equation*}
		Hence, $\Phi : Y_{T,K}\rightarrow Y_{T,K}$ is a contraction mapping.
		
		Here it remains to verify that there exists $K$ and $T$ such that (\ref{kt1})-(\ref{kt2}) and (\ref{kt3}) are satisfied. Obviously, taking $K>0$ small enough, we derive that (\ref{kt1}) and (\ref{kt3}) are satisfied. Next, by Lemma \ref{L-Heat semigroup properties1} and \eqref{Index}, one has
		\begin{equation*}
			t^{\alpha-\frac{1}{2}}\Vert A^{\alpha}e^{-tA}u_{0}\Vert_{2}
			=t^{\alpha-\frac{1}{2}}\Vert A^{\alpha-\frac{1}{2}}e^{-tA}A^{\frac{1}{2}}u_{0}\Vert_{2}\rightarrow0
		\end{equation*}
		and
		\begin{equation*}
			t^{\beta-\frac{1}{2}}\Vert A^{\beta}e^{-tA}u_{0}\Vert_{2}
			=t^{\beta-\frac{1}{2}}\Vert A^{\beta-\frac{1}{2}}e^{-tA}A^{\frac{1}{2}}u_{0}\Vert_{2}\rightarrow0
		\end{equation*}
		as $t\rightarrow0$ provided $u_{0}\in H^{1}_{0}(\Omega)$ and $\alpha,\beta>1/2$. Therefore, there exists $T>0$ small enough such that (\ref{kt2}) is satisfied. In both cases, we can assume that $K$ satisfies
		\begin{equation}\label{K}
			K\leq C\Vert u_{0}\Vert_{H^{1}_{0}}.
		\end{equation}
		Therefore, by applying Banach's Fixed Point Theorem, we show that $\Phi_{u_{0}}[u]$ has a unique fixed point $u$ in $Y_{T,K}$. Clearly, $u$ is a solution of problem \eqref{PI} in $[0,T]$.
		
		Next, similar to the proof of (\ref{c1}), we shall show that $u(t)\in C([0,T];H^{1}_{0}(\Omega))$. First, by (\ref{PI}) and standard computations, one has for $0<t_1\leq t_2\leq T$:
		\begin{equation}\label{e1}
			\begin{aligned}
				\Vert A^{\gamma}(u(t_{1})-u(t_{2}))\Vert_{2}
				\leq&\Vert A^{\gamma}(e^{-t_{1}A}u_{0}-e^{-t_{2}A}u_{0})\Vert_{2}\\
				&+\bigg\Vert A^{\gamma}\int_{0}^{t_{1}}e^{-(t_{1}-\tau)A}\left[ \left( \int_{\Omega}\frac{|u(y,\tau)|^{2^{\ast}_{\mu}}}{|x-y|^{\mu}}dy\right) |u(\tau)|^{2^{\ast}_{\mu}-2}u(\tau)\right]d\tau\\
				&-A^{\gamma}\int_{0}^{t_{2}}e^{-(t_{2}-\tau)A}\left[ \left(\int_{\Omega}\frac{|u(y,\tau)|^{2^{\ast}_{\mu}}}{|x-y|^{\mu}}dy\right) |u(\tau)|^{2^{\ast}_{\mu}-2}u(\tau)\right]d\tau\bigg\Vert_{2}\\
				\leq&\Vert A^{\gamma}e^{-t_{1}A}\left( I-e^{-(t_{2}-t_{1})A}\right) u_{0}\Vert_{2}
				+\bigg\Vert \int_{0}^{t_{1}}A^{\gamma}
				(e^{-(t_{1}-\tau)A}-e^{-(t_{2}-\tau)A})\\
				&\ \ \times\left[\left( \int_{\Omega}\frac{|u(y,\tau)|^{2^{\ast}_{\mu}}}{|x-y|^{\mu}}dy \right)  |u(\tau)|^{2^{\ast}_{\mu}-2}u(\tau)\right]d\tau\bigg\Vert_{2}\\
				&+\bigg\Vert\int_{t_{1}}^{t_{2}}A^{\gamma}e^{-(t_{2}-\tau)A}\left[ \left( \int_{\Omega}\frac{|u(y,\tau)|^{2^{\ast}_{\mu}}}{|x-y|^{\mu}}dy\right) |u(\tau)|^{2^{\ast}_{\mu}-2}u(\tau)\right]d\tau\bigg\Vert_{2}\\
				:=&A_{1}+A_{2}+A_{3}.
			\end{aligned}
		\end{equation}
		By the properties of the associated semigroup of contractions and Lemma \ref{L-Heat semigroup properties1}, we know that
		\begin{equation}\label{important}
			A_{1}\leq \Vert I-e^{-(t_{2}-t_{1})A}\Vert\Vert A^{\gamma}e^{-t_{1}A}u_{0}\Vert_{2}
			\leq Ct_{1}^{-(\gamma-\frac{1}{2})}\Vert I-e^{-(t_{2}-t_{1})A}\Vert \Vert u_{0}\Vert_{H^{1}_{0}}
		\end{equation}
		for $u_{0}\in H^{1}_{0}(\Omega)$.
		It follows from Lemma \ref{L-Heat semigroup properties1}, (\ref{nonlinear-estimate1}) and \eqref{key} that
		\begin{equation}\nonumber
			\begin{aligned}
				A_{2}\leq&\Vert I-e^{-(t_{2}-t_{1})A}\Vert\left\Vert \int_{0}^{t_{1}}A^{\gamma}e^{-(t_{1}-\tau)A}\left[ \left( \int_{\Omega}\frac{|u(y,\tau)|^{2^{\ast}_{\mu}}}{|x-y|^{\mu}}dy\right) |u(\tau)|^{2^{\ast}_{\mu}-2}u(\tau)\right]d\tau\right\Vert_{2}\\
				\leq&\Vert I-e^{-(t_{2}-t_{1})A}\Vert t_1^{\frac{1}{2}-\gamma}CK^{22^{\ast}_{\mu}-1}.
			\end{aligned}
		\end{equation}
		By the properties of semigroup, we obtain $A_{2}\rightarrow0$ as $t_{1}\rightarrow t_{2}$. Similarly, we have $A_{3}\rightarrow0$ as $t_{1}\rightarrow t_{2}$. Combine (\ref{e1}) and (\ref{important}), we derive that $u\in C((0,T];D(A^{\gamma}))$ for $\gamma \in[1/2,1)$ and $u\in C([0,T];H^{1}_{0}(\Omega))$.
		Consequently, we have completed the proof of $u\in C([0,T];H^{1}_{0}(\Omega))$.
		
		To show the uniqueness, let $u,v\in Y_{T,K}$ be two solutions of problem (\ref{PI}). Similar to (\ref{c4}), we derive
		\begin{equation*}
			\Vert\vert u-v\Vert\vert_{Y_{K,T}}\leq C\sup_{t\in (0,T]}t^{\frac{1}{2}-(2^{\ast}_{\mu}-1)(\beta-\frac{1}{2})-2^{\ast}_{\mu}(\alpha-\frac{1}{2})}(2K)^{2(2^{\ast}_{\mu}-1)}\Vert\vert u-v\Vert\vert_{Y_{K,T}},
		\end{equation*}
		which implies that $u=v$. Hence, the proof of uniqueness part is now complete. We now prove that $u_t\in L^2((0,T)\times \Omega)$. From assertion (ii) of Theorem 2  in \cite{Hoshino-Yamada}, we have $u_t\in L^2((t_0,T)\times \Omega)$ for any $t_0\in (0,T)$ and 
		\begin{equation}\label{energy-equality}
	\int_{t_0}^T\int_{\Omega}u_t\,dxdt=\left[J_\mu(u(t))\right]_{t_0}^T.
		\end{equation}
	Passing to the limit as $t_0\to 0^-$, we obtain $u_t\in L^2((0,T)\times \Omega)$.
		Finally, we show that this solution is a classical solution. The proof of this property is same as in \cite{Hoshino-Yamada}.
		Let $u\in Y_{T,K}$ be the solution of \eqref{PI} and $0<t<t+h\leq T$. It follows from \eqref{Integral Equation} that
		\begin{equation}\label{aaa1}
			\begin{aligned}
				\Vert A^{\gamma}\Phi_{u_{0}} [u]&(t+h)-A^{\gamma}\Phi_{u_{0}} [u](t)\Vert_{2}\\
				\leq&\Vert A^{\gamma}(e^{-hA}-I)e^{-tA}u_{0}\Vert_{2}
				+\int_{t}^{t+h}\Vert A^{\gamma}e^{-(t+h-\tau)A}f(u(\tau))\Vert_{2}d\tau\\
				&+\int_{0}^{t}\Vert A^{\gamma}(e^{-hA}-I)e^{-(t-\tau))A}f(u(\tau))\Vert_{2}d\tau
				=:B_{1}+B_{2}+B_{3}.
			\end{aligned}
		\end{equation}
		By \cite[Lemma 2.1$(ii)$]{Hoshino-Yamada} and Lemma \ref{L-Heat semigroup properties1}, we derive
		\begin{equation}\label{aaa2}
			B_{1}\leq Ch^{\delta}\Vert A^{\gamma+\delta}e^{-tA}u_{0}\Vert_{2}
			\leq Ch^{\delta}t^{-(\gamma+\delta-\frac{1}{2})}\Vert u_{0}\Vert_{H^{1}_{0}}
		\end{equation}
		for $\delta\in(0,1)$ and $u_{0}\in H^{1}_{0}(\Omega)$.
		Next, similar to the proof of \eqref{nonlinear-estimate1}, it follows from the definition of $Y_{T,K}$ that
		\begin{equation}\label{aaa3}
			B_{2}\leq C\int_{t}^{t+h} (t+h-\tau)^{-\gamma}e^{-\lambda(t+h-\tau)}\Vert f(u(\tau))\Vert_{2}d\tau
			\leq CK^{22^{\ast}_{\mu}-1}h^{1-\gamma}t^{-\frac{1}{2}}.
		\end{equation}
		Similar to \eqref{aaa3}, it follows from \cite[Lemma 2.1$(ii)$]{Hoshino-Yamada} and Lemma \ref{L-Heat semigroup properties1} that
		\begin{equation}\label{aaa4}
			B_{3}\leq Ch^{\delta}\int_{0}^{t} (t-\tau)^{-(\gamma+\delta)}e^{-\lambda(t-\tau)}\Vert f(u(\tau))\Vert_{2}d\tau
			\leq CK^{22^{\ast}_{\mu}-1}h^{\delta}t^{\frac{1}{2}-(\gamma+\delta)}
		\end{equation}
		for  $\delta\in(0,1-\gamma)$.
		Therefore, by \eqref{aaa1}-\eqref{aaa4}, one has
		\begin{equation}\label{aaa5}
			\begin{aligned}
				\Vert A^{\gamma}\Phi_{u_{0}} [u](t+h)-A^{\gamma}\Phi_{u_{0}} [u](t)\Vert_{2}
				\leq& Ch^{\delta}t^{-(\gamma+\delta-\frac{1}{2})}\Vert u_{0}\Vert_{H^{1}_{0}}+CK^{22^{\ast}_{\mu}-1}h^{1-\gamma}t^{-\frac{1}{2}}\\
				&+CK^{22^{\ast}_{\mu}-1}h^{\delta}t^{\frac{1}{2}-(\gamma+\delta)},
			\end{aligned}
		\end{equation}
		which implies that $A^{\gamma}\Phi_{u_{0}}[u](t)$ is H\"older continuous in $t\in(\varepsilon,T]$ for any $\varepsilon>0$ with $\gamma \in[1/2,1)$. Furthermore, by \eqref{b15} and \eqref{aaa5}, we get the H\"older continuity of $f(u(t))$ in $t\in(\varepsilon,T]$ for any $\varepsilon>0$. More precisely,
		\begin{equation}\label{cccc}
			\begin{aligned}
				\Vert f(u(t+h))-&f(u(t))\Vert_{2}\\
				\leq&C(2K)^{2(2^{\ast}_{\mu}-1)}h^{\delta}t^{-(\alpha+\delta-\frac{1}{2})}[\Vert u_{0}\Vert_{H^{1}_{0}}+K^{22^{\ast}_{\mu}-1}]+C(2K)^{42^{\ast}_{\mu}-3}h^{1-\alpha}t^{-\frac{1}{2}}\\
				&+C(2K)^{2(2^{\ast}_{\mu}-1)}h^{\delta}t^{-(\beta+\delta-\frac{1}{2})}[\Vert u_{0}\Vert_{H^{1}_{0}}+K^{22^{\ast}_{\mu}-1}]+C(2K)^{42^{\ast}_{\mu}-3}h^{1-\beta}t^{-\frac{1}{2}}
			\end{aligned}
		\end{equation}
		for $\delta\in(0,\min\{1-\alpha,1-\beta\})$ and $0<t<t+h\leq T$.
		As is well known, using parabolic regularity theory, the H\"older continuity of $f(u(t))$ with respect to $t\in(0,T]$ imply that the weak  solution $u$ is a classical solution for $t\in(0,T)$.
		
		Finally we prove \eqref{bbb5} and \eqref{bbb6}.
		For this purpose, we note that:
		\begin{equation}\label{bbb1}
			\begin{aligned}
				A\int_{0}^{t}e^{-(t-\tau)A}f(u(\tau))d\tau=&\int_{0}^{\frac{t}{2}}Ae^{-(t-\tau)A}f(u(\tau))d\tau\\
				+&\int_{\frac{t}{2}}^{t}Ae^{-(t-\tau)A}\bigg( f(u(\tau))-f(u(t))\bigg)d\tau\\
				+&[I-e^{-\frac{t}{2}A}]f(u(t))
				:=D_{1}+D_{2}+D_{3}.
			\end{aligned}
		\end{equation}
		By Lemma \ref{L-Heat semigroup properties1}, \eqref{b1} and \eqref{K} , we get
		\begin{equation}\label{bbb2}
			\begin{aligned}
				t^{\frac{1}{2}}\Vert D_{1}\Vert_{2}\leq&t^{\frac{1}{2}}\int_{0}^{\frac{t}{2}}\Vert Ae^{-(t-\tau)A}f(u(\tau))\Vert_{2}d\tau\\
				 \leq&Ct^{\frac{1}{2}}K^{22^{\ast}_{\mu}-1}\int_{0}^{\frac{t}{2}}(t-\tau)^{-1}\tau^{-(2^{\ast}_{\mu}-1)(\beta-\frac{1}{2})-2^{\ast}_{\mu}(\alpha-\frac{1}{2})}d\tau
				\leq C\Vert u_{0}\Vert_{H^{1}_{0}}
			\end{aligned}
		\end{equation}
		for some $C>0$.
		Similarly to the proof of \eqref{cccc}, it follows from \eqref{b1} and \eqref{K} that
		\begin{equation}\label{bbb3}
			t^{\frac{1}{2}}\Vert D_{2}\Vert_{2}\leq Ct^{\frac{1}{2}}\int_{\frac{t}{2}}^{t}(t-\tau)^{-1}\Vert f(u(\tau))-f(u(t))\Vert_{2}d\tau
			\leq C\Vert u_{0}\Vert_{H^{1}_{0}}.
		\end{equation}
		Furthermore, by \eqref{b1} and \eqref{K}, we also have
		\begin{equation}\label{bbb4}
			t^{\frac{1}{2}}\Vert D_{3}\Vert_{2}
			\leq Ct^{\frac{1}{2}-(2^{\ast}_{\mu}-1)(\beta-\frac{1}{2})-2^{\ast}_{\mu}(\alpha-\frac{1}{2})}K^{22^{\ast}_{\mu}-1}\leq C\Vert u_{0}\Vert_{H^{1}_{0}}.
		\end{equation}
		Therefore, by Lemma \ref{L-Heat semigroup properties1} and \eqref{bbb1}-\eqref{bbb4}, one has
		\begin{equation}\nonumber
			\begin{aligned}
				t^{\frac{1}{2}}\Vert Au(t)\Vert_{2}\leq& t^{\frac{1}{2}}\Vert Ae^{-tA}u_{0}\Vert_{2}
				+t^{\frac{1}{2}}\bigg\Vert A\int_{0}^{t}e^{-(t-\tau)A}f(u(\tau))d\tau\bigg\Vert_{2}\\
				\leq& C\Vert u_{0}\Vert_{H^{1}_{0}}
			\end{aligned}
		\end{equation}
		which implies that $\Vert Au(t)\Vert_{2}\leq Ct^{-\frac{1}{2}}\Vert u_{0}\Vert_{H^{1}_{0}}$ for $t\in(0,T]$.
		Furthermore, by the property of operator fractional powers, we complete the proof of \eqref{bbb5}.

		Next, the differentiation of \eqref{PI} with respect to $t$ provides that
	\begin{equation}\nonumber
			\begin{aligned}
				u_{t}(t)=&-Ae^{-tA}u_0+e^{-\frac{t}{2}A}f(u(t))-\int^{\frac{t}{2}}_{0}Ae^{-(t-\tau)A}f(u(\tau))d\tau\\
				&-\int_{\frac{t}{2}}^{t}Ae^{-(t-\tau)A}(f(u(\tau))-f(u(t)))d\tau.
			\end{aligned}
		\end{equation}
		Similar to the proof of \eqref{bbb5}, we show that \eqref{bbb6} holds for $\delta\in(0,\min\{1-\alpha,1-\beta\})$ and $0\leq \gamma<\delta$.
		Consequently, the proof is now completed.
	\end{proof}
	
	From the above Proposition, it follows that $u$ can be extended as a maximal solution in $(0,T_{\max})$ with $T<T_{\max}\leq \infty$ and $T_{\max}$ depending only on $\|u_0\|_{H^1_0(\Omega)}$ (since $T=T(\|u_0\|_{H^1_0(\Omega)})$). In the following Proposition, we investigate the blow-up pattern of solutions in case where $T_{\max}<\infty$.
	
	\begin{Prop}\label{P-Finite time}
			Let $u(t)$ be a maximal solution of \eqref{E} and $T_{\max}$ be the maximum existence time of $u$. If $T_{\max}<+\infty$, then $\lim_{t\rightarrow T_{\max}}\Vert u(t)\Vert_{\infty}=+\infty$
			and $\lim_{t\rightarrow T_{\max}}\Vert u(t)\Vert_{H^1_0(\Omega)}=+\infty$.
	\end{Prop}
	\begin{proof}
		If this not the case, then there exists one sequence $\left\lbrace t_{n}\right\rbrace _{n}$ such that $t_{n}\rightarrow T^{-}_{\max}$ and $\sup_{n}\Vert u(t_{n})\Vert_{\infty}<\infty$.
		It follows from the Hardy-Littlewood-Sobolev inequality and the fact $J_\mu(u(t_n))$ is non increasing in $n$ that $\sup_{n}\|u(t_n\|_{H^1_0(\Omega)}<\infty$.
		Then taking  $n$ large such that $t_n$ verifies $T_{\max}-t_n<\frac{T}{2}$ ($T$ as in Theorem \ref{thm-existence}),  there exists a unique solution $u_{n}:\ [0,T/2]\rightarrow H^{1}_{0}(\Omega)$ of \eqref{PI} with initial value $u(t_{n})$. By uniqueness, we know $u_{n}(t)=u(t+t_{n})$ for $t$ small. Define
		\begin{equation*}
			\bar{u}(t)=\left\{
			\begin{array}{ll}
				u(t),\ &t\in[0,t_{n}]\\
				u_{n}(t-t_{n}),\ &t\in[t_{n},t_{n}+T/2]
			\end{array}
			\right.
		\end{equation*}
		Then $\bar{u}$ is a weak solution to \eqref{E} on $[0,t_{n}+T/2]$ and $t_{n}+\frac{T}{2}>T_{\max}$, which contradicts that $u$ is a solution with maximal time interval $(0,T_{\max})$.		
		Similar contradiction can be derived if there exists a sequence $\left\lbrace t_{n}\right\rbrace _{n}$ such that $t_{n}\rightarrow T_{\max}^-$ and $\sup_{n}\Vert u(t_{n})\Vert_{H^{1}_{0}(\Omega)}<\infty$.
		The proof is thus completed.
	\end{proof}

	Next, multiplying problem \eqref{E} by $u$ and $u_{t}$ respectively and integrating over $\Omega$, we obtain the results with respect to energy identity and invariance of $W_{\mu}$ and $V_{\mu}$. We first infer the following Proposition whose proof is obtained by straightforward calculations.
	
	\begin{Prop}\label{P-Energy Identity}
		Let $u$ be a maximal solution to problem \eqref{E} on $[0,T_{\max})$.
		Then
		\begin{equation}\label{energy-identity1}
			\int_{0}^{t}\Vert u_{t}(\tau)\Vert^{2}_{2}d\tau+J_{\mu}(u(t))=J_{\mu}(u_{0})\ \text{on}\ [0,T_{\max}),
		\end{equation}
		and
		\begin{equation}\label{energy-identity2}
			\frac{1}{2}\frac{d}{dt}\Vert u(t)\Vert^{2}_{2}=-I_{\mu}(u(t))\ \text{on}\ [0,T_{\max}).
		\end{equation}
	\end{Prop}
	The next proposition deals with the invariance of $W_\mu$ and $V_\mu$ sets:
	\begin{Prop}\label{P-Invariant set}
		Let $u$ be a maximal solution of problem \eqref{E} on $[0,T_{\max})$.
		If there exists a $t_{0}\in[0,T_{\max})$ such that $u(t_{0})\in W_{\mu}$ (respectively $V_{\mu}$), then $u(t)\in W_{\mu}$ (respectively $V_{\mu}$) for any $t\in[t_{0},T_{\max})$.
	\end{Prop}
	\begin{proof}
		Let $u(t_{0})\in W_{\mu}$ for $t_{0}\in[0,T_{\max})$. It follows from Proposition \ref{P-Energy Identity} that $J_{\mu}(u(t))$ is non-increasing along the flow generated by problem \eqref{E}. Therefore, $u(t)\in W_{\mu}\cup V_{\mu}$ for any $t\in[t_{0},T_{\max})$.
		
		Next, we show that $u(t)\in W_{\mu}$ for any $t\in[t_{0},T_{\max})$. On the contrary, if there exists a $t_{1}\in[t_{0},T_{\max})$ such that $u(t_{1})\in V_{\mu}$, then there exists a $\tilde{t}_{1}\in[t_{0},t_{1})$ such that $I_{\mu}(u(\tilde{t}_{1}))=0$ provided $t\to I_{\mu}(u(t))$ is continuous on $[t_{0},T_{\max})$.
		
		Next, by the fact that $J_{\mu}(u)\geq m_{\mu}$ provided $I_{\mu}(u)=0$ and combining with (\ref{energy-identity1}), one has
		\begin{equation*}
			m_{\mu}\leq J_{\mu}(u(\tilde{t}_{1}))\leq J_{\mu}(u(t_{0})),
		\end{equation*}
		which contradicts with $J_{\mu}(u(t_{0}))<m_{\mu}$. Hence, $u(t)\in W_{\mu}$ for any $t\in[t_{0},T_{\max})$.
		
		Analogously, if $u(t_{0})\in V_{\mu}$ for $t_{0}\in[0,T_{\max})$, then we prove similarly that $u(t)\in V_{\mu}$ for any $t\in[t_{0},T_{\max})$. 
	\end{proof}

	\section{Proof of Theorems \ref{T-Existence potential well}--\ref{T-Global existence}}
	In this section, we shall verify the existence of potential well structure and give the proof of global existence and  solutions blow-up in finite time of problem \eqref{E}.
	
	\textbf{Proof of Theorem \ref{T-Existence potential well}.} For a nonnegative function $\phi\in L^{\infty}\cap H^{1}_{0}(\Omega)$ and nontrivial, we denote the (maximal) solution of problem \eqref{E} by $u\, :\,[0,T_{\max})\ni t\mapsto H^1_0(\Omega)$ with initial value $u(0)=u_{0}:=\lambda \phi$, where $\lambda>0$.
	
	Set
	\begin{equation}\nonumber
		\begin{aligned}
			\Lambda_{s}:=&\left\lbrace \lambda\in\mathbb{R}^{+}\ |\ u(t_{0})\in W_{\mu}\ \text{for some}\ t_{0}\in[0,T_{\max})\right\rbrace ,\\
			\Lambda_{us}:=&\left\lbrace \lambda\in\mathbb{R}^{+}\ |\ u(t_{0})\in V_{\mu}\ \text{for some}\ t_{0}\in[0,T_{\max})\right\rbrace ,\\
			\Lambda_{c}:=&\left\lbrace \lambda\in\mathbb{R}^{+}\ |\ u(t)\not\in W_{\mu}\cup V_{\mu}\ \text{for all}\ t\in[0,T_{\max})\right\rbrace.
		\end{aligned}
	\end{equation}
	Clearly, $\mathbb{R}^{+}=\Lambda_{s}\cup\Lambda_{us}\cup\Lambda_{c}$ thanks to Proposition \ref{P-Invariant set}. Furthermore, by the properties of $W_{\mu}$ and $V_{\mu}$, it is easy to show that $\Lambda_{s},\Lambda_{us}$ are nonempty sets in $\mathbb{R}^{+}$.
	
	Next, by the comparison principle (see \cite[Proposition 2]{Li-Liu}), we know that if $u_{0}\leq v_{0}$ a.e. in $\Omega$, then the solutions $u(t)$ and $v(t)$ of problem \eqref{E} with initial data $u_{0}$ and $v_{0}$ respectively satisfies $u(t)\leq v(t)$ for any $0\leq t < T_{\max}(v_{0})\leq  T_{\max}(u_{0})$, where $T_{\max}(v_{0})$,  $T_{\max}(u_{0})$ are the maximal existence time of $v(t)$ and $u(t)$ respectively. Since $0\in W_\mu$ and $\lambda \in V_\mu$ for large $\lambda>0$, $\Lambda_s$ and $\Lambda_{us}$ are non empty. Furthermore, using \cite[Theorems 1.4, 1.5]{Zhang-Yang} together with Proposition \ref{P-Invariant set}, we infer that
	\begin{equation}\label{g1}
		\Lambda_{s}\ \text{and}\ \Lambda_{us}\ \text{are ordered intervals in}\ \mathbb{R}^{+}.
	\end{equation}
	By Proposition \ref{P-Properties}, we know that $W_{\mu}$ and $V_{\mu}$ are open sets. For arbitrary $t>0$ but fixed, it follows from the openness of $W_{\mu},V_{\mu}$ and the continuity of
	\begin{equation}\label{g2}
		\lambda\in\mathbb{R}^{+}\mapsto\lambda\phi\in H^{1}_{0}(\Omega)\mapsto u(t)\in H^{1}_{0}(\Omega)
	\end{equation}
	that
	\begin{equation}\label{g3}
		\Lambda_{s}\ \text{and}\ \Lambda_{us}\ \text{are open sets in}\ \mathbb{R}^{+}.
	\end{equation}
	Therefore, by (\ref{g1})-(\ref{g3}), there exists $0<\lambda_{1}\leq\lambda_{2}<\infty$ such that
	\begin{equation*}
		\Lambda_{s}=(0,\lambda_{1}),\ \Lambda_{us}=(\lambda_{2},\infty)\ \text{and}\ \Lambda_{c}=[\lambda_{1},\lambda_{2}].
	\end{equation*}
	Consequently, the proof is complete.
	$\hfill{} \Box$
	
	\begin{lem}\label{Lemma 3.1}
		For each $u\in H^{1}_{0}(\Omega)\setminus\{0\}$, there exists a unique $\bar{\theta}>0$ such that $\bar{\theta}u\in \mathcal{N}$, where
		\begin{equation*}
			\mathcal{N}=\left\lbrace u\in H^{1}_{0}(\Omega)\setminus\{0\}\ |\ I_{\mu}(u)=0\right\rbrace
		\end{equation*}
		acts as the Nehari manifold associated to problem \eqref{E}.
		Furthermore, if $u\in V_{\mu}$, then $\bar{\theta}\in(0,1)$ and
		\begin{equation*}
			 A(u)
			\geq\frac{22^{\ast}_{\mu}}{2^{\ast}_{\mu}-1}m_{\mu}.
		\end{equation*}
	\end{lem}
	\begin{proof}
		By the definition of $I_{\mu}(u)$ in (\ref{Nehari-Energy}), we get
		\begin{equation*}
			I_{\mu}(\theta u)=A(\theta u)
			-B(\theta u)=\theta^{2}A(u)
			-\theta^{22^{\ast}_{\mu}}B(u).
		\end{equation*}
		Therefore, there exist a unique
		\begin{equation*}
			\bar{\theta}:=\left(\frac{A(u)}{B(u)}\right)^{1/(22^{\ast}_{\mu}-2)}
		\end{equation*}
		such that $\bar{\theta}u\in \mathcal{N}$.
		If $u\in V_{\mu}$, then we get that $A(u)<B(u)$, and so $\bar{\theta}\in(0,1)$.
		
		Furthermore, by the fact that $m_{\mu}=\inf_{v\in \mathcal{N}}J_{\mu}(v)$ and (\ref{Potential Deep}), we obtain
		\begin{equation*}
			m_{\mu}
			=\frac{2^{\ast}_{\mu}-1}{22^{\ast}_{\mu}}\inf_{v\in \mathcal{N}}A(v)\leq\frac{2^{\ast}_{\mu}-1}{22^{\ast}_{\mu}}A(v)
		\end{equation*}
		for $v\in\mathcal{N}$. Hence, one has $A(v)\geq\frac{22^{\ast}_{\mu}}{2^{\ast}_{\mu}-1}m_{\mu}$ for $v\in\mathcal{N}$.
		
		From above, for $u\in V_{\mu}$, there exists a unique $\bar{\theta}\in(0,1)$ such that $\bar{\theta}u\in \mathcal{N}$. Then
		\begin{equation*}
			A(u)
			\geq\bar{\theta}^{2}A(u)\geq\frac{22^{\ast}_{\mu}}{2^{\ast}_{\mu}-1}m_{\mu}>0.
		\end{equation*}
	The proof is now completed.
	\end{proof}
	
	\textbf{Proof of Theorem \ref{T-Blow up}.} In case that $J_{\mu}(u(t))<0$ for some $t\geq t_0$, the result of $T_{\max}<+\infty$ is classical, see \cite{Tsutsumi1} (or \cite[Theorem 1.5]{Zhang-Yang}). Using Propositions \ref{P-Finite time} and \ref{P-Invariant set} assertions in Theorem \ref{T-Blow up} hold in this case. Now we assume that $J_{\mu}(u(t))\geq0$ for any $t_0\leq t<T_{\max}$. Arguing by contradiction, we assume that $T_{\max}=+\infty$. It follows from (\ref{Energy}) and Proposition \ref{P-Invariant set} that
	\begin{equation}\nonumber
		\begin{aligned}
			\frac{1}{2}\frac{d}{dt}\Vert u(t)\Vert^{2}_{2}
			=&-A(u(t))
			+B(u(t))
			=-2J_{\mu}(u(t))+\frac{2^{\ast}_{\mu}-1}{2^{\ast}_{\mu}}B(u(t))\\
			\geq&-2J_{\mu}(u(t_{0}))+\frac{2^{\ast}_{\mu}-1}{2^{\ast}_{\mu}}B(u(t))
		\end{aligned}
	\end{equation}
	for $t\in[t_{0},+\infty)$.
	Set $\sigma:=1-\frac{J_{\mu}(u(t_{0}))}{m_{\mu}}>0$. Then, we get
	\begin{equation*}
		\frac{1}{2}\frac{d}{dt}\Vert u(t)\Vert^{2}_{2}
		\geq-2m_{\mu}(1-\sigma)+\frac{2^{\ast}_{\mu}-1}{2^{\ast}_{\mu}}B(u(t)),\ t\in[t_{0},+\infty).
	\end{equation*}
	By Lemma \ref{Lemma 3.1} and the definition of $V_{\mu}$, we get for any $t\geq t_0$
	\begin{equation*}
	B(u(t))>A(u(t))
		\geq\frac{22^{\ast}_{\mu}}{2^{\ast}_{\mu}-1}m_{\mu},
	\end{equation*}
	and then $2m_{\mu}(1-\sigma)<\frac{(2^{\ast}_{\mu}-1)(1-\sigma)}{2^{\ast}_{\mu}}B(u(t))$.
	Therefore, one has for any $t\geq t_0$
	\begin{equation}\label{i1}
		\frac{1}{2}\frac{d}{dt}\Vert u(t)\Vert^{2}_{2}
		\geq\frac{(2^{\ast}_{\mu}-1)\sigma}{2^{\ast}_{\mu}}B(u(t))
		\geq\frac{(2^{\ast}_{\mu}-1)\sigma}{2^{\ast}_{\mu}}A(u(t)).
	\end{equation}
	By (\ref{i1}) and the Poincar\'e inequality, we conclude that there exists $T_{1}>0$ such that $\lim_{t\rightarrow T_{1}}\Vert u(t)\Vert_{2}=+\infty$ which together with Sobolev embedding contradicts $T_{\max}=+\infty$. Again from Propositions \ref{P-Finite time} and \ref{P-Invariant set},  assertions of Theorem \ref{T-Blow up} hold.
	$\hfill{} \Box$
	
	Next, we prove Theorem \ref{T-Global existence} by a rescaling argument introduced in (\ref{Scaling}).
	By straightforward calculations, we infer the invariance of problem \eqref{E} and energy $J_{\mu}$ under this scaling. Precisely, one has
	
	\begin{Prop}\label{P-Scaling properties}
		Problem \eqref{E} is invariant under the change of variables given by (\ref{Scaling}). Namely, $v$ satisfies
		\begin{equation*}
			v_{s}-\Delta v=\left( \int_{\Omega}\frac{|v(z)|^{2^{\ast}_{\mu}}}{|y-z|^{\mu}}dy\right) |v|^{2^{\ast}_{\mu}-2}v,\ (y,s)\in\Omega\times[0,\eta]
		\end{equation*}
		for $\eta>0$ if and only if $u$ satisfies
		\begin{equation*}
			u_{t}-\Delta u=\left( \int_{\Omega_{\lambda}}\frac{|u(y)|^{2^{\ast}_{\mu}}}{|x-y|^{\mu}}dy\right) |u|^{2^{\ast}_{\mu}-2}u,\ (x,t)\in\Omega_{\lambda}\times[t_{0},t_{0}+\lambda^{-2}\eta]
		\end{equation*}
		where $\Omega_{\lambda}:=x_{0}+\lambda^{-1}\Omega$.
		Moreover,
		\begin{equation*}
			\left\Vert v_{s}\right\Vert_{L^{2}((0,\eta)\times\Omega)}
			=\left\Vert u_{t}\right\Vert_{L^{2}((t_{0},t_{0}+\lambda^{-2}\eta)\times\Omega_{\lambda})}
		\end{equation*}
		and
		\begin{equation*}
			\Vert \nabla v(s)\Vert_{2,\Omega}
			=\Vert \nabla u(t)\Vert_{2,\Omega_{\lambda}},\ \int_{\Omega}\left( \int_{\Omega}\frac{|v(z)|^{2^{\ast}_{\mu}}}{|y-z|^{\mu}}dy\right) |v|^{2^{\ast}_{\mu}}dy
			=\int_{\Omega_{\lambda}}\left( \int_{\Omega_{\lambda}}\frac{|u(y)|^{2^{\ast}_{\mu}}}{|x-y|^{\mu}}dy\right) |u|^{2^{\ast}_{\mu}}dx
		\end{equation*}
		and
		\begin{equation*}
			\Vert v(s)\Vert_{2,\Omega}=\lambda\Vert u(t)\Vert_{2,\Omega_{\lambda}}.
		\end{equation*}
		Furthermore, the energy $J_{\mu}$ is invariant under the rescaling \eqref{Scaling}, that is
		\begin{equation*}
			J_{\mu}(v(s))=J_{\mu}(u(t)).
		\end{equation*}
	\end{Prop}
	
	\begin{lem}\label{L3.5}
		Let $\left\lbrace u_{n}\right\rbrace\subset H^{1}_{0}(\Omega)\cap L^{\infty}(\Omega)$ be a bounded sequence satisfying
		\begin{equation}\label{L2}
			g_{n}:=\Delta u_{n}+\left( \int_{\Omega}\frac{|u_{n}(y)|^{2^{\ast}_{\mu}}}{|x-y|^{\mu}}dy\right) |u_{n}|^{2^{\ast}_{\mu}-2}u_{n}\rightarrow0\ \text{in}\ H^{-1}(\Omega)
		\end{equation}
		as $n\rightarrow\infty$. Then there exists $u\in H^{1}_{0}(\Omega)\cap L^\infty(\Omega)$ such that $u_{n}\rightharpoonup u$ weakly in $H^{1}_{0}(\Omega)$ and
		\begin{equation*}
			\text{either}\ u=0\ \text{or}\ A(u)\geq S^{\frac{2N-\mu}{N-\mu+2}}_{H,L}.
		\end{equation*}
	\end{lem}
	\begin{proof}
		From the definition of $J_{\mu}$ and (\ref{L2}), for any $\phi \in H^{1}_{0}(\Omega)$, one has
		\begin{equation}\label{L3}
			\begin{aligned}
				\vert\left\langle J^{\prime}_{\mu}(u_{n}),\phi\right\rangle \vert=&\left\vert \int_{\Omega}\Delta u_{n}\phi dx
				 +\int_{\Omega}\int_{\Omega}\frac{|u_{n}(y)|^{2^{\ast}_{\mu}}}{|x-y|^{\mu}}dy|u_{n}(x)|^{2^{\ast}_{\mu}-2}u_{n}(x)\phi dx\right\vert\\
				\leq& \Vert g_{n}\Vert_{H^{-1}(\Omega)}\Vert \phi\Vert_{H^1_0(\Omega)}\rightarrow0\ \text{as}\ n\rightarrow\infty,
			\end{aligned}
		\end{equation}
		where $J^{\prime}_{\mu}$ is the Fr\'echet derivative of $J_{\mu}$ at $u$. Taking $\phi =u_{n}$ in (\ref{L3}), we have
		\begin{equation}\label{L4}
			A(u_{n})-B(u_{n})=o(1)\ \text{as}\ n\rightarrow\infty.
		\end{equation}
		Next, since $\left\lbrace u_{n}\right\rbrace_{n\in \mathbb N}\subset H^{1}_{0}(\Omega)$ is a bounded sequence, then there exists $u\in H^{1}_{0}(\Omega)\cap L^\infty(\Omega)$ such that up to a subsequence (still denoted $\{u_n\}_{n\in \mathbb N}$):
		\begin{equation*}
			u_{n}\rightharpoonup u\ \text{weakly in}\ H^{1}_{0}(\Omega).
		\end{equation*}
		Hence, we have
		\begin{equation}\label{L6}
			A(u)\leq\liminf_{n\rightarrow\infty} A(u_{n}).
		\end{equation}
		Since $L^\infty$ boundedness of $\{u_n\}$, by the regularity theory, there exists a subsequence (still denoted by $\left\lbrace u_{n}\right\rbrace $) such that $u_{n}\rightarrow u$ in $C(\bar{\Omega})$ as $n\rightarrow\infty$ and
		\begin{equation}\label{f3}
			u_{n}\rightarrow u\ \text{in}\ L^{2^{\ast}}
		\end{equation}
		as $n\rightarrow\infty$.
		Furthermore, by the Hardy-Littlewood-Sobolev inequality, we also know that
		\begin{equation*}
			B(u_{n})=B(u)+o(1).
		\end{equation*}
		Combining with (\ref{L4}) and (\ref{L6}), one has
		\begin{equation}\nonumber
				A(u)\leq\liminf_{n\rightarrow\infty} A(u_{n})\\
				 =\liminf_{n\rightarrow\infty} B(u_{n})+o(1)=B(u)+o(1).
		\end{equation}
		Furthermore, if $u\not=0$, it follows from (\ref{SHL}) that
		\begin{equation*}
			A(u)\geq S^{\frac{2N-\mu}{N-\mu+2}}_{H,L}.
		\end{equation*}
	The proof is now completed.
	\end{proof}
	
	\textbf{Proof of Theorem \ref{T-Global existence}.}
	First, from Proposition \ref{P-Invariant set}, we know that $u(t)\in W_{\mu}$ for $t\in[t_{0},T_{\max})$ provided $u(t_{0})\in W_{\mu}$ for some $t_{0}\in [0,T_{\max})$. Furthermore, from \eqref{bbb5} and \eqref{bbb6} together with Propositions \ref{P-Finite time} and again \ref{P-Invariant set},  we directly derive that if $u(t_0)\in W_\mu$, then $T_{\max}=+\infty$. Next, without loss of generality, we assume that $u(t)\in W_{\mu}\setminus\{0\}$ for $t_0\leq t<\infty$, since $u\equiv\,0$ is an obvious case.
	By the definition of $W_{\mu}$, we derive
	\begin{equation}\label{d1}
		J_{\mu}(u)=\frac{1}{2}A(u)
		-\frac{1}{22^{\ast}_{\mu}}
		B(u)
		\geq\left(\frac{1}{2}-\frac{1}{22^{\ast}_{\mu}} \right)B(u)>0,\ \forall t\in[t_{0},\infty).
	\end{equation}
	Next, by Proposition \ref{P-Energy Identity} and (\ref{d1}), there  exists $c\geq0$ such that
	\begin{equation}\label{d11}
		\lim_{t\rightarrow \infty}J_{\mu}(u(t))=c.
	\end{equation}
	We claim that
	\begin{equation}\label{d10}
		\lim_{t\rightarrow \infty}\Vert u(t)\Vert_{\infty}=0.
	\end{equation}
	On the contrary, there exists $\delta\in(0,\infty]$ such that
	\begin{equation}\label{d2}
		\limsup_{t\rightarrow \infty}\Vert u(t)\Vert_{\infty}=\delta.
	\end{equation}
	Therefore, there exists a sequence $\left\lbrace t_{n}\right\rbrace $ with $t_{n}\rightarrow \infty$ such that $\Vert u(t_{n})\Vert_{\infty}\rightarrow\delta$ as $n\rightarrow\infty$ and
	\begin{equation*}
		\Vert u(t_{n})\Vert_{\infty}\geq\frac{1}{2}\sup_{t\in[0,t_{n}]}\Vert u(t)\Vert_{\infty}
	\end{equation*}
	for any $n\in\mathbf{N}$. And let $\left\lbrace x_{n}\right\rbrace $ be a sequence such that
	\begin{equation*}
		\frac{1}{2}\Vert u(t_{n})\Vert_{\infty}
		\leq|u(x_{n},t_{n})|.
	\end{equation*}
	Set
	\begin{equation*}
		\lambda_{n}:=\Vert u(t_{n})\Vert^{\frac{2}{N-2}}_{\infty}.
	\end{equation*}
According to \eqref{Scaling}, we define
	\begin{equation*}
		u_{n}(y,s):=\lambda_{n}^{-\frac{N-2}{2}}u(x,t),
	\end{equation*}
	where $y=\lambda_{n}(x-x_{n})$ and $s=\lambda_{n}^{2}(t-t_{n})$. Then letting $\Omega_{\lambda_{n}}:=\lambda_{n}(\Omega-x_{n})$, it is easy to see that $\lambda_{n}\rightarrow\delta^{2/(N-2)}$ as $n\rightarrow\infty$, there exists $C>0$ such that
	\begin{equation}\label{d3}
		\sup_{s\in[-1,0]}\Vert u_{n}(s)\Vert_{\infty,\Omega_{\lambda_{n}}}
		\leq\sup_{t\in[0,t_{n}]}\Vert \lambda_{n}^{-\frac{N-2}{2}}u(t)\Vert_{\infty,\Omega}\,
		\leq C
	\end{equation}
	and
	\begin{equation}\label{d3-1}
		\big\vert u_{n}(y,s)|_{y=0,s=0}\big\vert
		=\big\vert \lambda_{n}^{-\frac{N-2}{2}}u(x_{n},t_{n})\big\vert
		\geq\frac{1}{2}.
	\end{equation}
	By Proposition \ref{P-Energy Identity} and \ref{P-Scaling properties}, we also have
	\begin{equation}\label{d3-2}
		\begin{aligned}
		\left\Vert\frac{\partial u_{n}}{\partial s}\right\Vert^2_{L^{2}((-1,0)\times\Omega_{\lambda_{n}})}
			=&\left\Vert\frac{\partial u}{\partial t}\right\Vert^2_{L^{2}((t_{n}-1/\lambda_{n}^{2},t_{n})\times\Omega)}\\
			=&J_{\mu}(u(t_{n}-1/\lambda_{n}^{2}))-J_{\mu}(u(t_{n})).
		\end{aligned}
	\end{equation}
	If $\delta=\infty$, we have
	\begin{equation*}
		t_{n}-1/\lambda_{n}^{2}\rightarrow \infty
	\end{equation*}
	provided $\lambda_{n}\rightarrow\delta^{2/(N-2)}$ and $t_{n}\rightarrow \infty$. Furthermore, by (\ref{d3-2}) and \eqref{d11}, we have
	\begin{equation}\label{d3-3}
		\left\Vert\frac{\partial u_{n}}{\partial s}\right\Vert_{L^{2}((-1,0)\times\Omega_{\lambda_{n}})}\rightarrow0.
	\end{equation}
	If contrary $\delta<\infty$, 
	one has
	\begin{equation*}
		t_{n}-1/\lambda_{n}^{2}=t_{n}-1/\delta^{4/(N-2)}+o(1)=\infty.
	\end{equation*}
	By (\ref{d3-2}) again, one sees that (\ref{d3-3}) holds also.
	
	By the fact that $t_{n}\rightarrow \infty$ and $t_{n}-1/\lambda_{n}^{2}\rightarrow \infty$, we have $t_{n}+s/\lambda_{n}^{2}\in(t_{0},\infty)$ for $n$ large enough. Further, from \eqref{energy-identity1}, Proposition  \ref{P-Invariant set} and $u(t_{0})\in W_{\mu}$, we derive that $u(t_{n}+s/\lambda_{n}^{2})\in W_{\mu}$ and $J_{\mu}(u(t_{n}+s/\lambda_{n}^{2}))\leq J_{\mu}(u(t_{0}))$. Therefore, it follows from  Proposition \ref{P-Scaling properties} that
	\begin{equation}\label{AAA}
		\begin{aligned}
			m_{\mu}>J_{\mu}(u(t_{0}))\geq& J_{\mu}(u(t_{n}+s/\lambda_{n}^{2}))
			\geq\left(\frac{1}{2}-\frac{1}{22^{\ast}_{\mu}} \right)A(u(t_{n}+s/\lambda_{n}^{2}))\\
			=&\left(\frac{1}{2}-\frac{1}{22^{\ast}_{\mu}} \right)A(u_{n}(s)),\ \forall s\in[-1,0],
		\end{aligned}
	\end{equation}
	which implies that $\{u_{n}(s)\}_{n\in\mathbb N}$ is bounded in $H^{1}_{0}(\Omega)$ for all $s\in[-1,0]$.
	
	Next, by (\ref{d3}), we have $\left\lbrace u_{n}\right\rbrace_{n\in \mathbb N} $ is bounded in $L^{\infty}(-1,0;L^{\infty}(\Omega_{\lambda_{n}}))$. By $L^{p}$ estimate for parabolic operator (see \cite{Lieberman}), we obtain that $\left\lbrace u_{n}\right\rbrace $ is a bounded sequence in $W^{2,1}_{p,loc}((-1,0)\times\Omega_{\lambda_{n}})$ for $p$ large enough. Therefore, it follows from the Sobolev Embedding Theorem that $\left\lbrace u_{n}\right\rbrace_{n\in \mathbb N} $ is a bounded sequence in $C^{0,\alpha;0,\alpha/2}_{loc}((-1,0)\times\Omega_{\lambda_{n}})$ for all $\alpha\in(0,1)$. By Ascoli-Arzela Theorem, there exists a subsequence, still denoted by $\left\lbrace u_{n}\right\rbrace_{n\in \mathbb N}$, and $v$ such that
	\begin{equation}\label{d7}
		u_{n}\rightarrow v\ \text{in}\ C_{loc}((-1,0]\times\Omega_{\lambda_{n}})
	\end{equation}
	as $n\rightarrow\infty$.
	By (\ref{d3-1}) and (\ref{d7}), we infer that $\big|v(y,s)|_{y=0,s=0}\big|\geq1/2$. Thus, $v\not\equiv0$.
	
	It follows from (\ref{d3}) and (\ref{d3-3}) that for some $\kappa\in[-1,0]$
	\begin{equation*}
		\Delta u_{n}(\kappa)+\left( \int_{\Omega_{\lambda_{n}}}\frac{|u_{n}(z,\kappa)|^{p}}{|y-z|^{\mu}}dz\right) |u_{n}(\kappa)|^{2^{\ast}_{\mu}-2}u_{n}(\kappa)=\frac{\partial u_{n}(\kappa)}{\partial s}\rightarrow0\ \text{in}\ L^{2}(\Omega_n)
	\end{equation*}
	as $n\rightarrow\infty$ and
	\begin{equation*}
		\Vert u_{n}(\kappa)\Vert_{\infty,\,\Omega_{\lambda_n}}\leq C
	\end{equation*}
	for all $n\in\mathbf{N}$.
	
	Furthermore, since $\{u_{n}(s)\}$ is bounded in $H^{1}_{0}(\Omega)$ for all $s\in[-1,0]$, there exists a $w\in H^{1}_{0}$ such that
	\begin{equation}\label{d8}
		u_{n}(\kappa)\rightharpoonup w\ \text{weakly in}\ H^{1}_{0}.
	\end{equation}
	And by (\ref{d7}) and (\ref{d8}), one has $w=v\not\equiv0$.
	Furthermore, by Lemma \ref{L3.5}, we have
	\begin{equation}\label{d9}
		A(w)\geq S^{\frac{2N-\mu}{N-\mu+2}}_{H,L}.
	\end{equation}
	Therefore, by (\ref{AAA}) and (\ref{d9}), we deduce
	\begin{equation*}
		S^{\frac{2N-\mu}{N-\mu+2}}_{H,L}
		\leq A(w)
		\leq\liminf_{n\rightarrow\infty}A(u_{n})< S^{\frac{2N-\mu}{N-\mu+2}}_{H,L},
	\end{equation*}
	which is a contradiction. Thus, we obtain that $\lim_{t\rightarrow \infty}\Vert u(t)\Vert_{\infty}=0$. 
	
	Next, by Proposition \ref{P-Energy Identity}, we know that $u(t)\in W_{\mu}$ for all $t\in[t_{0},\infty)$, combining with (\ref{energy-identity2}), we deduce that
 \begin{equation}\label{L2est}
		\frac{1}{2}\frac{d}{dt}\Vert u(t)\Vert^{2}_{2}=-I_{\mu}(u(t))<0.
	\end{equation}
	Therefore,
	\begin{equation}\label{d13-1}
			\Vert u(t)\Vert^{2}_{2}\rightarrow 0
	\end{equation}
	as $t\rightarrow\infty$. Furthermore, by (\ref{d10}) and (\ref{d13-1}), one has
	\begin{equation}\label{d14}
		u(t)\rightarrow0\ \text{in}\ L^{r}(\Omega),\ \forall r\in [2,\infty]. 
	\end{equation}
	By (\ref{d11}) , we deduce that
	\begin{equation}\label{d12}
		\lim_{t\rightarrow+\infty}J_{\mu}(u(t))=c\geq0.
	\end{equation}
	From \eqref{L2est} and the boundedness of $\|u(t)\|_{2}$, we infer that there exists $\left\lbrace t_{n}\right\rbrace $ with $t_{n}\rightarrow\infty$ as $n\rightarrow\infty$ such that
	\begin{equation}\label{d13}
		A(u(t_{n}))
		=B(u(t_{n}))+o(1)\ \text{as}\ n\rightarrow\infty.
	\end{equation}
	By (\ref{d12}) and (\ref{d13}), we derive
	\begin{equation*}
		 c=\lim_{n\rightarrow+\infty}J_{\mu}(u(t_{n}))=\lim_{n\rightarrow+\infty}\left(\frac{1}{2}-\frac{1}{22^{\ast}_{\mu}} \right) B(u(t_{n})).
	\end{equation*}
	By the Hardy-Littlewood-Sobolev inequality and (\ref{d14}) with $r=2^{\ast}$, we know that
	\begin{equation*}
		\lim_{t\rightarrow\infty}B(u(t))=0.
	\end{equation*}
	Thus, $c=0$ and $\lim_{t\rightarrow+\infty}J_{\mu}(u(t))=0$.
	Furthermore, we have
	\begin{equation*}
		A(u(t))
		=2J_{\mu}(u(t))+\frac{1}{2^{\ast}_{\mu}}B(u(t))\rightarrow0.
	\end{equation*}
The proof is completed.
	$\hfill{} \Box$
	
	\begin{Rem}
		The argument of $A(u(t))\rightarrow0$ as $t\rightarrow\infty$ can also be proved by the method of \cite{Zhang-Yang} (or see \cite{Ikehata-Suzuki}) and $A(u(t))$ has decay with exponential type.
	\end{Rem}

	\section{Proof of Theorem \ref{T-Blow up infinite}}
	
	In this section, we prove that the  blow-up in infinite time of the global solutions occurs under suitable conditions.  First, assuming that $T_{\max}=+\infty$, we claim that
	\begin{equation}\label{y1}
		J_{\mu}(u(t))\geq 0
	\end{equation}
	for all $t\geq0$. Indeed, by the definition $J_{\mu}(u)$ and $I_{\mu}(u)$, one has
	\begin{equation*}
		J_{\mu}(u)=\frac{1}{2}A(u)
		-\frac{1}{22^{\ast}_{\mu}}B(u)\geq\frac{1}{2}I_{\mu}(u).
	\end{equation*}
	If there exist $t_{0}>0$ such that $J_{\mu}(u(t_{0}))<0$, then $u(t_{0})\in V_{\mu}$, and then $T_{\max}<+\infty$ by Theorem \ref{T-Blow up}, from which we get a contradiction.
	
	By Proposition \ref{P-Energy Identity} and (\ref{y1}), we know that
	\begin{equation*}
		\int_{0}^{\infty}\Vert u_{t}(\tau)\Vert^{2}_{2}d\tau\leq J_{\mu}(u_{0})<+\infty.
	\end{equation*}
	Furthermore, there exists a sequence $\left\lbrace t_{n}\right\rbrace $ with $t_{n}\rightarrow+\infty$ as $n\rightarrow+\infty$ such that
	\begin{equation}\label{y2}
		\Vert u_{t}(t_{n})\Vert_{2}\rightarrow0\ \text{as}\ n\rightarrow+\infty.
	\end{equation}
	Moreover, from  \eqref{d13-1} (see also\cite{Cazenave-Lions}), we also have
	\begin{equation}\label{y2-1}
		\Vert u(t_{n})\Vert_{2}\leq C
	\end{equation}
	for some constant $C>0$.
	Then, we state and prove the following two Lemmas.
	\begin{lem}\label{L4.1}
		Let $u$ be a global (classical)  solution of problem \eqref{E} with $T_{\max}=\infty$. Let $\left\lbrace t_{n}\right\rbrace $ be a sequence satisfying (\ref{y2}). Then there exists a constant $\tilde{C}>0$ such that
		\begin{equation*}
			A(u(t_{n}))\leq \tilde{C},\
			B(u(t_{n}))\leq \tilde{C}
		\end{equation*}
		and
		\begin{equation*}
			I_{\mu}(u(t_{n}))\rightarrow0\ \text{as}\ n\rightarrow\infty.
		\end{equation*}
		Furthermore, we also have
		\begin{equation}\label{x}
			\lim_{n\rightarrow+\infty}A(u(t_{n}))
			=\lim_{n\rightarrow+\infty}B(u(t_{n}))=C_{0},
		\end{equation}
		where $C_{0}$ is defined in (\ref{C_{0}}).
	\end{lem}
	\begin{proof}
		Multiplying problem \eqref{E} by $u(t)$ and integrating over $\Omega$, we show that
		\begin{equation}\label{x1}
			(u_{t}(t),u(t))=-A(u(t))
			+B(u(t)).
		\end{equation}
		Since $\left\lbrace t_{n}\right\rbrace $ satisfies (\ref{y2}), it follows from \eqref{y2-1} and (\ref{x1}) that
		\begin{equation*}
			\left| -A(u(t_{n}))
			+B(u(t_{n}))\right|
			\leq\Vert u(t_{n})\Vert_{2}\Vert u_{t}(t_{n})\Vert_{2}\rightarrow0
		\end{equation*}
		as $n\rightarrow+\infty$, which implies that
		\begin{equation}\label{x2}
			A(u(t_{n}))
			=B(u(t_{n}))+o(1)\ \text{as}\ n\rightarrow+\infty.
		\end{equation}
		Next, by the definition of $J_{\mu}$ and (\ref{x2}), one has
		\begin{equation}\label{x3}
			J_{\mu}(u(t_{n}))=\frac{1}{2}A(u(t_{n}))
			-\frac{1}{22^{\ast}_{\mu}}
			B(u(t_{n}))
			=\frac{2^{\ast}_{\mu}-1}{22^{\ast}_{\mu}}A(u(t_{n}))+o(1).
		\end{equation}
		By Proposition \ref{P-Energy Identity}, we also have
		\begin{equation}\label{x4}
			J_{\mu}(u(t_{n}))=J_{\mu}(u_{0})-\int_{0}^{t_{n}}\Vert u_{t}(\tau)\Vert^{2}_{2}d\tau\leq J_{\mu}(u_{0}).
		\end{equation}
		Therefore, by (\ref{x2})-(\ref{x4}), there exists a constants $\tilde{C}>0$ such that
		\begin{equation*}
			A(u(t_{n}))\leq \tilde{C}\ \text{and}\ B(u(t_{n}))\leq \tilde{C}.
		\end{equation*}
		And by (\ref{x2}) and definition of $I_{\mu}$, we also have
		\begin{equation*}
			I_{\mu}(u(t_{n}))\rightarrow0\ \text{as}\ n\rightarrow\infty.
		\end{equation*}
		Here, we set
		\begin{equation*}
			C_{1}:=\lim_{n\rightarrow+\infty}A(u(t_{n}))\ \text{and}\ C_{2}:=\lim_{n\rightarrow+\infty}B(u(t_{n})).
		\end{equation*}
		It follows from the definition of $C_{0}$ and (\ref{x2})-(\ref{x3}) that $C_{1}=C_{2}=C_{0}$. The proof is completed.
	\end{proof}
	
	\begin{lem}\label{L4.2}
		Suppose that $C_{0}=0$. Then for some $t_{1}>0$, there holds
		\begin{equation*}
			\int_{t_{1}}^{+\infty}\Vert \Delta u(t)\Vert^{2}_{2}dt<+\infty.
		\end{equation*}
	\end{lem}
	\begin{proof}
		If $C_{0}=0$, by Lemma \ref{L4.1} and the property of $W_{\mu}$ in Proposition \ref{P-Properties}, one has $u(t_{n})\in W_{\mu}$ for $n$ large enough. And then it follows from Theorem \ref{T-Global existence} that
		\begin{equation}\label{u0}
			\Vert \nabla u(t)\Vert_{2}\rightarrow0\ \text{as}\ t\rightarrow\infty.
		\end{equation}
		Next, by problem $\eqref{E}$ and similar to the proof in Lemma \ref{L-Nonlinear estimate}, we derive for a.e. $t\geq 0$
		\begin{equation}\label{u1}
			\begin{aligned}
				\Vert \Delta u(t)\Vert^{2}_{2}
				\leq&2\left( \Vert u_{t}(t)\Vert^{2}_{2}+\int_{\Omega}|f(u(t))|^{2}dx\right)\\
				\leq&2\left( \Vert u_{t}(t)\Vert^{2}_{2}
				+\Vert A^{\alpha}u(t)\Vert^{22^{\ast}_{\mu}}_{2}\Vert A^{\beta}u(t)\Vert^{2(2^{\ast}_{\mu}-1)}_{2}\right)
			\end{aligned}
		\end{equation}
	where $\alpha$, $\beta$ are given in \eqref{Index} with $\theta$ large.
		Furthermore, by interpolation and (\ref{u1}), one has
		\begin{equation}\label{u2}
			\Vert \Delta u(t)\Vert^{2}_{2}
			\leq 2\Vert u_{t}(t)\Vert^{2}_{2}
			+C\Vert \Delta u(t)\Vert^{2}_{2}\Vert \nabla u(t)\Vert^{4(2^{\ast}_{\mu}-1)}_{2}
		\end{equation}
		provided $\alpha,\beta\in(1/2,1)$ defined by (\ref{Index}).
		By (\ref{u0}), there exists a $t_{1}>0$ such that
		\begin{equation}\label{u3}
			C\Vert \nabla u(t)\Vert^{4(2^{\ast}_{\mu}-1)}_{2}<1/2
		\end{equation}
		for any $t>t_{1}$. Combining (\ref{u2}) with (\ref{u3}), we have
		\begin{equation}\label{u4}
			\frac{1}{4}\Vert \Delta u(t)\Vert^{2}_{2}
			\leq \Vert u_{t}(t)\Vert^{2}_{2}
		\end{equation}
		for any $t>t_{1}$.
		Therefore, it follows from (\ref{C_{0}}), (\ref{energy-identity1}) and (\ref{u4}) that
		\begin{equation*}
			\int_{t_{1}}^{+\infty}\Vert \Delta u(t)\Vert^{2}_{2}dt
			\leq C\int_{0}^{\infty}\Vert u_{t}\Vert^{2}_{2}dt
			=CJ_{\mu}(u_{0}),
		\end{equation*}
		provided $C_{0}=0$.The proof is completed.
	\end{proof}
	
	\textbf{Proof of Theorem \ref{T-Blow up infinite}.}
	$(i)\Rightarrow(ii)$. On the contrary, there exists $t_{0}\in[0,+\infty)$ such that $J_{\mu}(u(t_{0}))<m_{\mu}$, then we know that $u(t_{0})\in W_{\mu}$. Indeed, by Proposition \ref{P-Properties}, we know that $u(t_{0})\in W_{\mu}\cup V_{\mu}$. Furthermore, by Theorem \ref{T-Blow up}, if $u(t_{0})\in V_{\mu}$, then $T_{\max}<+\infty$, which is a contradiction.
	
	Next, by Theorem \ref{T-Global existence}, one has $u(t)\in W_{\mu}$ and $\Vert \nabla u(t)\Vert_{2}\rightarrow0$ as $t\rightarrow\infty$. Furthermore, by (\ref{SHL}), we also have
	\begin{equation*}
		\lim_{t\rightarrow\infty}B(u(t))=0,
	\end{equation*}
	and so $\lim_{t\rightarrow+\infty}J_{\mu}(u(t))=0$. By the definition of $C_{0}$ in (\ref{C_{0}}), we can deduce that $C_{0}=0$, which contradicts with the assumption.
	
	$(ii)\Rightarrow(iii)$. It is obviously.
	
	$(iii)\Rightarrow(iv)$. Assume that $u(t)\notin W_{\mu}\cup V_{\mu}$ for all $t\in[0,\infty)$. Then from  the continuity  of $J_{\mu}(u)$, we infer that
	\begin{equation*}
		\lim_{t\rightarrow+\infty}J_{\mu}(u(t))\geq m_{\mu},
	\end{equation*}
	and it follows from (\ref{energy-identity1}) that
	\begin{equation*}
		\int_{0}^{\infty}\Vert u_{t}(t)\Vert_{2}^{2}dt
		\leq J_{\mu}(u_{0})-m_{\mu}<+\infty.
	\end{equation*}
	Arguing by contradiction, let us assume that $\displaystyle\liminf_{t\rightarrow+\infty}\Vert u(t)\Vert_{\infty}<+\infty$. Choose $C>0$ and $t_{n}\rightarrow\infty$ such that
	$\Vert u(t_{n})\Vert_{\infty}<C$ for all $n=1,2,\cdots$, and furthermore, by parabolic regularity, there exists $\bar{t}>0$ with $t_{n+1}>t_{n}+2\bar{t}$ such that
	\begin{equation}\label{dd}
		\sup_{t\in(t_{n},t_{n}+\bar{t})}\Vert u(t)\Vert_{\infty}\leq C
	\end{equation}
	for all $n=1,2,\cdots$ and for some $C>0$. 
	
	Next, since
	\begin{equation*}
		\sum_{n}\int_{t_{n}}^{t_{n}+\bar{t}}\Vert u_{t}(t)\Vert_{2}^{2}dt
		\leq\int_{0}^{+\infty}\Vert u_{t}(t)\Vert_{2}^{2}dt<+\infty,
	\end{equation*}
	we have
	\begin{equation*}
		\lim_{n\rightarrow+\infty}\int_{t_{n}}^{t_{n}+\bar{t}}\Vert u_{t}(t)\Vert_{2}^{2}dt=0,
	\end{equation*}
	and so there exists $\bar{t}_{n}\in[t_{n},t_{n}+\bar{t}]$ such that $\lim_{n\rightarrow+\infty}\Vert u_{t}(\bar{t}_{n})\Vert_{2}=0$.
	By (\ref{dd}) and parabolic regularity, we show that $\left\lbrace u(\bar{t}_{n})\right\rbrace $ is compact in $C^{2}(\bar{\Omega})$. From star-shapedness of $\Omega$, there is no non-trivial stationary solution to problem \eqref{E} and it follows that
	\begin{equation*}
		\lim_{n\rightarrow+\infty}\Vert u(\bar{t}_{n})\Vert_{\infty}=0.
	\end{equation*}
	Furthermore, using parabolic regularity again, we know that $u(\bar{t}_{n})\in W_{\mu}$ for $n$ large, which is a contradiction. Consequently, the proof of $(iii)\Rightarrow(iv)$ is complete.
	
	$(iv)\Rightarrow(i)$. On the contrary, we assume that $C_{0}=0$. By Lemma \ref{L4.2}, we show that
	\begin{equation}\label{z1}
		\liminf_{t\rightarrow+\infty}\Vert \Delta u(t)\Vert_{2}=0.
	\end{equation}
	In fact, if there exists $C>0$ such that $\Vert \Delta u(t)\Vert_{2}>C$, then we derive a contradiction with
	$\int_{t_{1}}^{+\infty}\Vert \Delta u(t)\Vert^{2}_{2}<+\infty$.
	
	For $N=3$, by (\ref{z1}) and the embedding of $H^{2}(\Omega)$ into $L^{\infty}(\Omega)$, one has
	\begin{equation}\label{z2}
		\liminf_{t\rightarrow+\infty}\Vert u(t)\Vert_{\infty}=0,
	\end{equation}
	which is a contradiction with the assumption $(iv)$.
	$\hfill{} \Box$
	
	\begin{cor}
		Suppose that $\Omega$ is a star-shaped domain. If $u(t)$ is a time-global solution, then either
		$(i)$ $u(t_{0})\in W_{\mu}$ for some $t_{0}\geq0$ and $\lim_{t\rightarrow+\infty}\Vert u(t)\Vert_{\infty}=\lim_{t\rightarrow+\infty}\Vert \nabla u(t)\Vert_{2}=0$ and $\lim_{t\rightarrow+\infty}J_{\mu}(u(t))=0$
		or $(ii)$ $u(t)\not\in W_{\mu}\cup V_{\mu}$ for all $t\geq0$ and $\lim_{t\rightarrow+\infty}\Vert u(t)\Vert_{\infty}=+\infty$, $0<\liminf_{t\rightarrow+\infty}\Vert\nabla u(t)\Vert_{2}<+\infty$ and $\lim_{t\rightarrow+\infty}J_{\mu}(u(t))>0$.
	\end{cor}

	\section{Proof of Theorem \ref{T-Global bounds} and Corollary \ref{C-Global bounds}}
	
	In this section, we deal with $L^{\infty}$-global boundedness of the global solutions. Let $u$ be a global solution of problem \eqref{E}. Let
	\begin{equation*}
		c:=\lim_{t\rightarrow\infty}J_{\mu}(u(t)).
	\end{equation*}
	Then we have the following results.
	
	\begin{lem}\label{L5.1}
		Suppose that $J_{\mu}$ satisfies the $(PS)_{c}$-condition. Then
		\begin{equation*}
			A(u(t))\rightarrow\frac{22^{\ast}_{\mu}}{2^{\ast}_{\mu}-1}c\ \text{and}\ B(u(t))\rightarrow\frac{22^{\ast}_{\mu}}{2^{\ast}_{\mu}-1}c
		\end{equation*}
		as $t\rightarrow\infty$.
	\end{lem}
	\begin{proof}
		By Proposition \ref{P-Energy Identity} and (\ref{y1}), we know that $J_{\mu}(u(t))$ is nonincreasing in $t$ and is bounded from below, thus there exists a sequence $\left\lbrace \bar{t}_{n}\right\rbrace $ with $\bar{t}_{n}\rightarrow\infty$ as $n\rightarrow\infty$ such that
		\begin{equation*}
			\frac{d}{dt}J_{\mu}(u(\bar{t}_{n}))\rightarrow0,\ \text{as}\ n\rightarrow\infty.
		\end{equation*}
		Using Proposition \ref{P-Energy Identity} again, we have 
		\begin{equation}\label{t1}
			\|u_t(\bar{t}_n)\|_2\to 0
		\end{equation}
		as $n\rightarrow\infty$.
		
		We claim that
		\begin{equation}\label{t2}
			A(u(\bar{t}_{n}))=\frac{22^{\ast}_{\mu}}{2^{\ast}_{\mu}-1}c+o(1)\ \text{and}\ B(u(\bar{t}_{n}))=\frac{22^{\ast}_{\mu}}{2^{\ast}_{\mu}-1}c+o(1).
		\end{equation}
		Indeed, by the definition of $J_{\mu}$ and (\ref{t1}), for any $\phi \in H^{1}_{0}(\Omega)$, one has
		\begin{equation}\label{t0}
			\begin{aligned}
				\vert\left\langle J^{\prime}_{\mu}(u(\bar{t}_{n})),\phi\right\rangle \vert=
				&\bigg\vert \int_{\Omega}\Delta u(\bar{t}_{n})\phi dx
				 +\int_{\Omega}\int_{\Omega}\frac{|u(y,\bar{t}_{n})|^{2^{\ast}_{\mu}}}{|x-y|^{\mu}}dy|u(x,\bar{t}_{n})|^{2^{\ast}_{\mu}-2}u(x,\bar{t}_{n})\phi dx\bigg\vert\\
				=&\int_{\Omega}u_{t}(\bar{t}_{n})\phi dx
				\leq \Vert u_{t}(\bar{t}_{n})\Vert_{2}\Vert \phi\Vert_{2}\rightarrow0\ \text{as}\ n\rightarrow\infty.
			\end{aligned}
		\end{equation}
		Taking $\phi =u(\bar{t}_{n})$, we have
		\begin{equation}\label{t3}
			A(u(\bar{t}_{n}))-B(u(\bar{t}_{n}))=o(1)\ \text{as}\ n\rightarrow\infty.
		\end{equation}
		Furthermore, we derive
		\begin{equation*}
			A(u(\bar{t}_{n}))=\frac{22^{\ast}_{\mu}}{2^{\ast}_{\mu}-1}c+o(1)\ \text{and}\ B(u(\bar{t}_{n}))=\frac{22^{\ast}_{\mu}}{2^{\ast}_{\mu}-1}c+o(1)
		\end{equation*}
		since
		\begin{equation}\label{t4}
			c+o(1)=J_{\mu}(u(\bar{t}_{n}))=\frac{1}{2}A(u(\bar{t}_{n}))-\frac{1}{22^{\ast}_{\mu}}B(u(\bar{t}_{n})).
		\end{equation}
		Consequently, the proof of the claim is complete.
		
		Now let us suppose that the conclusion does not hold, then there exists a sequence $\left\lbrace \tilde{t}_{n}\right\rbrace $ with $\tilde{t}_{n}\rightarrow\infty$ as $n\rightarrow\infty$ and $\sigma\not=0$ such that
		\begin{equation}\label{t5}
			A(u(\tilde{t}_{n}))=\frac{22^{\ast}_{\mu}}{2^{\ast}_{\mu}-1}c+\sigma+o(1).
		\end{equation}
		
		We claim that there exists a sequence $\left\lbrace t_{n}\right\rbrace $ with $t_{n}\rightarrow\infty$ and $\bar{\sigma}$ with
		\begin{equation}\label{t6}
			\bar{\sigma}\leq\sigma\,, \bar{\sigma}<\left[(2^{\ast}_{\mu})^{\frac{N-2}{2N-\mu}}S_{H,L} \right] ^{\frac{2N-\mu}{N-\mu+2}}\ \text{and}\ \bar{\sigma}\not=0
		\end{equation}
		such that
		\begin{equation}\label{t7}
			A(u(t_{n}))=\frac{22^{\ast}_{\mu}}{2^{\ast}_{\mu}-1}c+\bar{\sigma}+o(1).
		\end{equation}
		Indeed, if $\sigma<0$, then taking $t_{n}=\tilde{t}_{n}$ and $\bar{\sigma}=\sigma$ and we derive the claim. If $\sigma>0$, it follows from (\ref{t2}), (\ref{t5}) and the Mean Value Theorem that there exists $t_{n}\in (\bar{t}_{n},\tilde{t}_{n})$ (without loss of generality, we can assume $\bar{t}_{n}<\tilde{t}_{n}$) satisfies (\ref{t7}) for any $\bar{\sigma}$ with (\ref{t6}).
		
		Next, it follows from (\ref{t7}) that
		\begin{equation*}
			c+o(1)=J_{\mu}(u(t_{n}))=\frac{1}{2}\left( \frac{22^{\ast}_{\mu}}{2^{\ast}_{\mu}-1}c+\bar{\sigma}\right) -\frac{1}{22^{\ast}_{\mu}}B(u(t_{n}))+o(1).
		\end{equation*}
		Furthermore, we derive
		\begin{equation}\label{t8}
			B(u(t_{n}))=\frac{22^{\ast}_{\mu}}{2^{\ast}_{\mu}-1}c+2^{\ast}_{\mu}\bar{\sigma}+o(1).
		\end{equation}
		By (\ref{t7}), $\left\lbrace u(t_{n})\right\rbrace $ is a bounded sequence in $H^{1}_{0}(\Omega)$, and then up to a subsequence there exists $\bar{u}\in H^{1}_{0}(\Omega)$ such that
		\begin{equation}\label{t9}
			u(t_{n})\rightharpoonup\bar{u}\ \text{in}\ H^{1}_{0}(\Omega)\mbox{ as }n\to\infty
		\end{equation}
		and from compact imbedding, we have
		\begin{equation}\label{t10}
			u(t_{n})\rightarrow\bar{u}\ \text{in}\ L^{2}(\Omega)\mbox{ as }n\to\infty.
		\end{equation}
		
		Next, by Proposition \ref{P-Energy Identity}, one has
		\begin{equation}\label{t11}
			\int_{t_{n}}^{t_{n}+1}\Vert u_{t}(\tau)\Vert^{2}_{2}d\tau
			=J_{\mu}(u(t_{n}))-J_{\mu}(u(t_{n}+1))\rightarrow c-c=0.
		\end{equation}
		Thus, there exist $\eta=\eta_n\in [0,1]$ such that $\Vert u_{t}(t_{n}+\eta)\Vert_{2}\rightarrow0$ as $n\rightarrow\infty$. Using (\ref{t0}), we know that $\left\lbrace u(t_{n}+\eta)\right\rbrace $ is a $(PS)_{c}$ sequence. Furthermore, since $J_\mu$ satisfies the $(PS)_{c}$-condition, there exists $w\in H^{1}_{0}(\Omega)$ such that
		\begin{equation}\label{t13}
			u(t_{n}+\eta)\rightarrow w\ \text{in}\ H^{1}_{0}(\Omega)
		\end{equation}
		as $n\rightarrow\infty$ and
		\begin{equation}\label{t14}
			A(w)= B(w)=\frac{22^{\ast}_{\mu}}{2^{\ast}_{\mu}-1}c.
		\end{equation}
		It follows from (\ref{t10}) and (\ref{t13}) that
		\begin{eqnarray}\label{conv}
			\Vert \bar{u}-w\Vert_{2}
			&\leq\Vert \bar{u}-u(t_{n})\Vert_{2}+\Vert u(t_{n})-u(t_{n}+\eta)\Vert_{2}+\Vert u(t_{n}+\eta)-w\Vert_{2}\nonumber\\
		&	=o(1)+\Vert u(t_{n})-u(t_{n}+\eta)\Vert_{2}.
		\end{eqnarray}
	Moreover, 
		\begin{equation*}
			\Vert u(t_{n})-u(t_{n}+\eta)\Vert_{2}\leq \int_{t_n}^{t_n+\eta}\Vert u_t(\tau)\Vert_2d\tau\leq \left(\int_{t_{n}}^{t_{n}+1}\Vert u_{t}(\tau)\Vert^{2}_{2}d\tau\right)^{\frac{1}{2}}=o(1)
		\end{equation*}
		which implies  together with \eqref{conv} that $\bar{u}=w$. Furthermore, by (\ref{t14}), we derive that
		\begin{equation}\label{t15}
			A(\bar{u})=B(\bar{u})=\frac{22^{\ast}_{\mu}}{2^{\ast}_{\mu}-1}c,
		\end{equation}
		and since $u(t_{n})\rightharpoonup\bar{u}$ in $H^{1}_{0}(\Omega)$
		\begin{eqnarray}\label{t16}
			A(u(t_{n}))=&A(\bar{u})+A(u(t_{n})-\bar{u})+2\left\langle \nabla(u(t_n)-\bar{u}),\nabla\bar{u}\right\rangle\nonumber\\
		&=\frac{22^{\ast}_{\mu}}{2^{\ast}_{\mu}-1}c+A(u(t_{n})-\bar{u})+2\left\langle \nabla(u(t_n)-\bar{u}),\nabla\bar{u}\right\rangle\nonumber\\
		&=\frac{22^{\ast}_{\mu}}{2^{\ast}_{\mu}-1}c+A(u(t_{n})-\bar{u})+o(1).
		\end{eqnarray}
		By Brezis-Lieb Lemma of Hartree type and (\ref{t15}), we also have
	\begin{equation}\label{t17}
			B(u(t_{n}))=B(\bar{u})+B(u(t_{n})-\bar{u})+o(1)
			=\frac{22^{\ast}_{\mu}}{2^{\ast}_{\mu}-1}c+B(u(t_{n})-\bar{u})+o(1).
		\end{equation}
		Next, it follows from (\ref{t7})-(\ref{t8}) and (\ref{t16})-(\ref{t17}) that
		\begin{equation}\label{t18}
			A(u(t_{n})-\bar{u})=\bar{\sigma}+o(1)
		\end{equation}
		and
		\begin{equation}\label{t19}
			B(u(t_{n})-\bar{u})=2^{\ast}_{\mu}\bar{\sigma}+o(1).
		\end{equation}
		By (\ref{SHL}), we have
		\begin{equation}\label{t20}
			\left( B(u(t_{n})-\bar{u})\right) ^{\frac{N-2}{2N-\mu}}S_{H,L}\leq A(u(t_{n})-\bar{u})
		\end{equation}
		where $S_{H,L}$ is the best constant in the sense of the Hardy-Littlewood-Sobolev inequality. Therefore, by (\ref{t18})-(\ref{t20}), we derive
		\begin{equation*}
			\bar{\sigma}\geq\left[(2^{\ast}_{\mu})^{\frac{N-2}{2N-\mu}}S_{H,L} \right] ^{\frac{2N-\mu}{N-\mu+2}}
		\end{equation*}
		which contradicts (\ref{t6}). The proof is now completed.
	\end{proof}
	
	\textbf{Proof of Theorem \ref{T-Global bounds}.}
	$(i)\Rightarrow(ii)$. Arguing by contradiction, one has that there exists a sequence $\left\lbrace t_{n}\right\rbrace $ with $t_{n}\rightarrow\infty$ and $\left\lbrace x_{n}\right\rbrace\subset\Omega$ such that
	\begin{equation}\label{s1}
		\Vert u(t_{n})\Vert_{\infty}\rightarrow\infty\ \text{as}\ n\rightarrow\infty,\ \sup_{t\in[0,t_{n}]}\Vert u(t)\Vert_{\infty}=\Vert u(t_{n})\Vert_{\infty}
	\end{equation}
	and
	\begin{equation}\label{s2}
		\vert u(x_{n},t_{n})\vert\geq\frac{\Vert u(t_{n})\Vert_{\infty}}{2}.
	\end{equation}
	Let
	\begin{equation*}
		\lambda_{n}:=\Vert u(t_{n})\Vert^{\frac{2}{N-2}}_{\infty}.
	\end{equation*}
	and
	\begin{equation}\label{s2-1}
		v_{n}(y,s):=\lambda_{n}^{-\frac{N-2}{2}}u(x,t),
	\end{equation}
	where $y=\lambda_{n}(x-x_{n})$ and $s=\lambda_{n}^{2}(t-t_{n})$. Then, by (\ref{s1})-(\ref{s2-1}), it is easy to see that $\lambda_{n}\rightarrow\infty$ as $n\rightarrow\infty$ and
	\begin{equation}\label{s3}
		\sup_{s\in[-1,0]}\Vert v_{n}(s)\Vert_{\infty}
		\leq\Vert v_{n}(0)\Vert_{\infty}
			=1
	\end{equation}
	and
	\begin{equation}\label{s3-1}
		\vert v_{n}(y,s)|_{y=0,s=0}\vert
		=\vert \lambda_{n}^{-\frac{N-2}{2}}u(x_{n},t_{n})\vert
		\geq\frac{1}{2}.
	\end{equation}
	Furthermore, there exists $\bar{x}\in \bar{\Omega}$ such that up to an extraction of a subsequence $x_{n}\rightarrow \bar{x}$. Following \cite{Gidas-Spruck}, we then considered two separate cases.
	
	$(I)$ The case $\bar{x}\in \Omega$. We infer that $y\in\lambda_{n}(\Omega-x_{n})\rightarrow\mathbb{R}^{N}$.
	Let
	\begin{equation*}
		B_{x}(\bar{x},\xi):=\left\lbrace x\in\mathbb{R}^{N}\ |\ |x-\bar{x}|<\xi\right\rbrace,\
		B_{x}(x_{n},R/\lambda_{n}):=\left\lbrace x\in\mathbb{R}^{N}\ |\ |x-x_{n}|<R/\lambda_{n}\right\rbrace.
	\end{equation*}
	Clearly, since $\bar{x}\in\Omega$, we know that $B_{x}(\bar{x},\xi)\subset\Omega$ for $\xi$ small enough. We also know that $B_{x}(x_{n},R/\lambda_{n})\subset B_{x}(\bar{x},\xi)$ for $n$ large enough provided $x_{n}\rightarrow\bar{x}$ and $\lambda_{n}\rightarrow\infty$.
	
	Similar to (\ref{d7}), it follows from (\ref{s3}) that there exists a subsequence (still denoted by $\left\lbrace v_{n}\right\rbrace $) and a $v$ such that
	\begin{equation}\label{s4}
		v_{n}\rightarrow v\ \text{in}\ C_{loc}((-1,\epsilon)\times\mathbb{R}^{N})).
	\end{equation}
	Next, by Propositions \ref{P-Energy Identity} and \ref{P-Scaling properties}, we have
	\begin{equation}\nonumber
		\begin{aligned}
			\left\Vert\frac{\partial v_{n}}{\partial s}\right\Vert^{2}_{L^{2}((-1,\epsilon)\times\mathbb{R}^{N})}
			=&\left\Vert\frac{\partial u}{\partial t}\right\Vert^{2}_{L^{2}((t_{n}-1/\lambda_{n}^{2},t_{n}+\epsilon/\lambda_{n}^{2})\times\Omega)}\\
			=&J_{\mu}(u(t_{n}-1/\lambda_{n}^{2}))-J_{\mu}(u(t_{n}+\epsilon/\lambda_{n}^{2}))\rightarrow0,
		\end{aligned}
	\end{equation}
	which implies that $v$ is independent of $s$ provided the same argument as \eqref{conv}.  By (\ref{s3-1}) and (\ref{s4}), we know that $\big\vert v(y)|_{y=0}\big\vert\geq1/2$. Furthermore, there exists $R>0$ small enough such that
	\begin{equation}\label{s5}
		\int_{B_{y}(0,R)}\int_{B_{y}(0,R)}\frac{|v(y)|^{2^{\ast}_{\mu}}|v(z)|^{2^{\ast}_{\mu}}}{|y-z|^{\mu}}dydz>0,
	\end{equation}
	where $B_{y}(0,R):=\left\lbrace y\in\mathbb{R}^{N}\ |\ |y|<R\right\rbrace$.
	
	On the other hand, by (\ref{s4}), Proposition \ref{P-Scaling properties} and the Hardy-Littlewood-Sobolev inequality, we derive that
	\begin{equation}\label{s6}
		\begin{aligned}
			\int_{B_{y}(0,R)}\int_{B_{y}(0,R)}\frac{|v(y)|^{2^{\ast}_{\mu}}|v(z)|^{2^{\ast}_{\mu}}}{|y-z|^{\mu}}&dydz
			 =\int_{B_{y}(0,R)}\int_{B_{y}(0,R)}\frac{|v_{n}(y,0)|^{2^{\ast}_{\mu}}|v_{n}(z,0)|^{2^{\ast}_{\mu}}}{|y-z|^{\mu}}dydz+o(1)\\
			 =&\int_{B_{x}(x_{n},R/\lambda_{n})}\int_{B_{x}(x_{n},R/\lambda_{n})}\frac{|u(x,t_{n})|^{2^{\ast}_{\mu}}|u(z,t_{n})|^{2^{\ast}_{\mu}}}{|x-z|^{\mu}}dxdz+o(1)\\
			 \leq&\int_{B_{x}(\bar{x},\xi)}\int_{B_{x}(\bar{x},\xi)}\frac{|u(x,t_{n})|^{2^{\ast}_{\mu}}|u(z,t_{n})|^{2^{\ast}_{\mu}}}{|x-z|^{\mu}}dxdz+o(1).
		\end{aligned}
	\end{equation}
	
	We claim that there exists $\tilde{u}$ such that
	\begin{equation}\label{s7}
		 \int_{B_{x}(\bar{x},\xi)}\int_{B_{x}(\bar{x},\xi)}\frac{|u(x,t_{n})|^{2^{\ast}_{\mu}}|u(z,t_{n})|^{2^{\ast}_{\mu}}}{|x-z|^{\mu}}dxdz\rightarrow \int_{B_{x}(\bar{x},\xi)}\int_{B_{x}(\bar{x},\xi)}\frac{|\tilde{u}(x)|^{2^{\ast}_{\mu}}|\tilde{u}(z)|^{2^{\ast}_{\mu}}}{|x-z|^{\mu}}dxdz
	\end{equation}
	as $n\rightarrow\infty$.
	Indeed, by Lemma \ref{L5.1}, there exists $\tilde{u}\in H^{1}_{0}(\Omega)$ such that
	\begin{equation*}
		u(t_{n})\rightharpoonup\tilde{u}\ \text{in}\ H^{1}_{0}(\Omega)
	\end{equation*}
	and by weak lower semicontinuity, we also have
	\begin{equation}\label{}
		B(\tilde{u})\leq\liminf_{n\rightarrow\infty} B(u(t_{n}))
		=\frac{22^{\ast}_{\mu}}{2^{\ast}_{\mu}-1}c+o(1).
	\end{equation}
	Similar to the proof of (\ref{t15}), we prove that $B(\tilde{u})=\frac{22^{\ast}_{\mu}}{2^{\ast}_{\mu}-1}c$.
	Therefore, we have $B(u(t_{n}))\rightarrow B(\tilde{u})$, which implies (\ref{s7}).
	It follows from (\ref{s6}) and (\ref{s7}) that
	\begin{equation*}
		\int_{B_{y}(0,R)}\int_{B_{y}(0,R)}\frac{|v(y)|^{2^{\ast}_{\mu}}|v(z)|^{2^{\ast}_{\mu}}}{|y-z|^{\mu}}dydz=0
	\end{equation*}
	since
	\begin{equation*}
		 \int_{B_{x}(\bar{x},\xi)}\int_{B_{x}(\bar{x},\xi)}\frac{|\tilde{u}|^{2^{\ast}_{\mu}}|\tilde{u}|^{2^{\ast}_{\mu}}}{|x-z|^{\mu}}dxdz\rightarrow0\ \text{as}\ \xi\rightarrow0,
	\end{equation*}
	which contradicts (\ref{s5}).
	
	$(II)$ The case $\bar{x}\in\partial\Omega$. We may assume that $\partial \Omega$ is contained in the hyperplane $x_{N}=0$.
	Define
	\begin{equation*}
		\bar{\kappa}:=\limsup_{n\rightarrow+\infty}\lambda_{n}dist(x_{n}, \partial\Omega)\geq 0.
	\end{equation*}
	
We first claim that $\bar{\kappa}\geq C_{1}>0$ for some $C_{1}$. Indeed, by parabolic $L^{p}$ theory up to the boundary, we get $|\nabla_{y} v_{n}|$ is uniformly bounded. In particular,
	\begin{equation*}
		|v_{n}(0)-v_{n}(0^{\prime},\bar{\kappa})|\leq C\bar{\kappa}
	\end{equation*}
	where $x^{\prime}:=x_{1},\cdots,x_{N-1}$. Furthermore, it follows from \eqref{s2-1} that
	\begin{equation*}
		1-\lambda_{n}^{-\frac{N-2}{2}}\sup u(x_{n}^{\prime},x_{n}+\bar{\kappa}/\lambda_{n})\leq C\bar{\kappa}
	\end{equation*}
	By \eqref{E} and $\lambda_{n}\rightarrow\infty$ as $n\rightarrow\infty$, we have $\sup u(x_{n}^{\prime},x_{n}+\bar{\kappa}/\lambda_{n})$ is bounded, furthermore, we show that $\bar{\kappa}$ is bounded from below for some $C_{1}>0$. Consequently, the claim holds.
	If $\bar\kappa=\infty$, we argue as in case (I) to derive contradiction. Suppose that $0<\bar\kappa<\infty$.  Then $y\in\lambda_{n}(\Omega-x_{n})\rightarrow\tilde{\mathbb{R}}^{N}$, where $\tilde{\mathbb{R}}^{N}:=\left\lbrace y=(y_{1},y_{2},\cdots,y_{N})\in\mathbb{R}^{N}\ |\ y_{N}>-\bar{\kappa}\right\rbrace $.
	Similarly, there exist a subsequence (still denoted by $\left\lbrace v_{n}\right\rbrace $) and  $v$ such that $v_{n}\rightarrow v\ \text{in}\ C_{loc}((-1,\epsilon)\times\tilde{\mathbb{R}}^{N}))$ for some $\epsilon >0$. Therefore, $v_{n}(0)\rightarrow v$ uniformly in $B_{y}(0,\bar{\kappa}/2)$. Furthermore, by (\ref{s3-1}), we have
	\begin{equation*}
		 \int_{B_{y}(0,\bar{\kappa}/2)}\int_{B_{y}(0,\bar{\kappa}/2)}\frac{|v(y)|^{2^{\ast}_{\mu}}|v(z)|^{2^{\ast}_{\mu}}}{|y-z|^{\mu}}dydz>0.
	\end{equation*}
	Then again we can argue  as in the case $\bar{x}\in\Omega$ and get the contradiction.
	
	$(ii)\Rightarrow(i)$. Let $t_{n}\rightarrow\infty$ as $n\rightarrow\infty$ be a sequence satisfies $\lim_{n\rightarrow+\infty}J_{\mu}(u(t_{n}))=c$ and
	\begin{equation}\label{f1}
		J_{\mu}^{\prime}(u(t_{n}))\rightarrow0\ \text{in}\ (H^{1}_{0}(\Omega))^{\ast}
	\end{equation}
	as $n\rightarrow\infty$. We shall show that $\left\lbrace u(t_{n})\right\rbrace $ admits a convergent subsequence in $H^{1}_{0}(\Omega)$ (with norm topology).
	
	By (\ref{f1}), we know that
	\begin{equation}\label{f2}
		A(u(t_{n}))
		-B(u(t_{n}))=o(1).
	\end{equation}
	Since $\sup_{t\in[0,+\infty)}\Vert u(t)\Vert_{\infty}<\infty$, by parabolic regularity, there exist a subsequence (still denoted by $\left\lbrace t_{n}\right\rbrace $) and $u$ such that $u(t_{n})\rightarrow u$ in $C(\bar{\Omega})$ as $n\rightarrow\infty$ and
	\begin{equation}\label{f3}
		u(t_{n})\rightarrow u\ \text{in}\ L^{2}\ \text{and}\ L^{2^{\ast}}
	\end{equation}
	as $n\rightarrow\infty$.
	Furthermore, by the Hardy-Littlewood-Sobolev inequality, we also know that
	\begin{equation}\label{f4}
		B(u(t_{n}))=B(u)+o(1).
	\end{equation}
	It follows from Proposition \ref{P-Energy Identity} and (\ref{f4}) that there exists a constant $C>0$ such that
	\begin{equation*}
		A(u(t_{n}))
		=2J_{\mu}(u(t_{n}))+\frac{1}{2^{\ast}_{\mu}}B(u(t_{n}))<C,
	\end{equation*}
	which implies that
	\begin{equation}\label{f5}
		u(t_{n})\rightharpoonup u\ \text{weakly in}\ H^{1}_{0}(\Omega)
	\end{equation}
	as $n\rightarrow\infty$.
	
	Next, by (\ref{f1}) and (\ref{f5}), it follows that
	\begin{equation*}
		o(1)=\left\langle J_{\mu}^{\prime}(u(t_{n})),u\right\rangle =\left\langle J_{\mu}^{\prime}(u),u\right\rangle
		=A(u)-B(u),
	\end{equation*}
	which implies that
	\begin{equation}\label{f6}
		A(u)=B(u).
	\end{equation}
	Therefore, by (\ref{f2}), (\ref{f4}) and $(\ref{f6})$, we obtain
	\begin{equation*}
		A(u(t_{n}))=B(u(t_{n}))+o(1)=B(u)+o(1)=A(u)+o(1)
	\end{equation*}
	from which we get $u(t_n)\to u$ in $H^1_0(\Omega)$ as $n\to\infty$ and the proof is completed. \qed
	
	\textbf{Proof of Corollary \ref{C-Global bounds}.}  Since $c<\frac{N-\mu+2}{2(2N-\mu)}S^{\frac{2N-\mu}{N-\mu+2}}_{H,L}$, by the fact that $J_{\mu}$ satisfies the $(PS)_{c}$ condition in \cite{Gao-Yang}. Then we complete the proof by appying Theorem \ref{T-Global bounds}. \qed
	
	\subsection*{Acknowledgements} Minbo Yang is partially supported by NSFC (11971436, 12011530199) and ZJNSF (LZ22A010001). V.D.~R\u adulescu was supported by the grant ``Nonlinear Differential Systems in Applied Sciences" of
		the Romanian Ministry of Research, Innovation and Digitization, within PNRR-III-C9-2022-I8/22.
	
	\bibliographystyle{plain}

\begin{thebibliography}{99}
		\bibitem{Badra-Bal-Giacomoni} M. Badra, K. Bal, J. Giacomoni, A singular parabolic equation: existence, stabilization, J. Differential Equations, 252(2012), no. 9, 5042--5075.
		
		\bibitem{Brezis-Cazenave} H. Brezis, T. Cazenave, A nonlinear heat equation with singular initial data, J. d’Anal. Math., 68(1996), 277--304.
		
		\bibitem{Cazenave-Lions} T. Cazenave, P. L. Lions, Solutions globales d’equations de la chaleur semilineaires, Comm. Partial Differential Equations, 9(1984), 955--978.
		
		\bibitem{Cazenave-Weissler} T. Cazenave, F. B. Weissler, The Cauchy problem for the critical nonlinear Schr$\ddot{o}$dinger equation in $H^{s}$, Nonlinear Anal. 14(1990) 807--836.
		
		\bibitem{Du-Yang} L. Du, M. Yang, Uniqueness and Nondegeneracy of solutions for a critical Nonlocal equation, Discrete and Continuous Dynamical Systems, 39(2019), 5847--5866.
		
		\bibitem{Ebihara-Nakao}  Y. Ebihara, M. Nakao, T. Nambu, On the existence of global classical solution of initial-boundary value problem for $\Box u-u^{3}=f$, Pacific J. Math., 60(1975), 63--70.
		
		\bibitem{Fila-Souplet-Weissler} M. Fila, P. Souplet, F.B. Weissler, Linear and nonlinear heat equations in $L_{\delta}^{q}$ spaces and universal bounds for global solutions, Math. Ann., 320(2001), 87--113.
		
		\bibitem{Gidas-Spruck} B. Gidas, J. Spruck, A priori bounds for positive solutions of nonlinear elliptic equations, Commun. Partial Differ. Equations,  6(1981), 883--901.
		
		\bibitem{Giga1} Y. Giga, A bound for global solutions of semilinear heat equations, Comm. Math. Phys., 103(1986), 415--421.
		
		\bibitem{Gourley} S. A. Gourley, Travelling front solutions of a nonlocal Fisher equation, J. Math. Biol., 41(2000) 272--284.
		
		\bibitem{Gao-Yang} F. Gao, M. Yang, The Brezis-Nirenberg type critical problem for the nonlinear Choquard equation, Science China Mathematics, 61(2018), 1219--1242.
		
		\bibitem{Henry} D. Henry, Geometric Theory of Semilinear Parabolic Equations, Lecture Notes in Math., 840(1981), Springer-Verlag.
		
		\bibitem{Hoshino-Yamada} H. Hoshino and Y. Yamada, Solvability and smoothing effect for semilinear parabolic equations, Funk. Ekva., 34(1991), 475--494.
		
		\bibitem{Ikehata} R. Ikehata, A note on the global solvability of solutions to some nonlinear wave equations with dissipative terms,  Differential and Integral Equations, 8(1995), 607--616.
		
		\bibitem{Ishii} H. Ishii, Asymptotic stability and blowing-up of solutions of some nonlinear equations, J. Diff. Eq., 26(1917), 291--319.
		
		\bibitem{Ishiwata} M. Ishiwata, Existence of a stable set for some nonlinear parabolic equation involving critical Sobolev exponent, Discrete Contin. Dyn. Syst., (2005), 443--452.
		
		\bibitem{Ishiwata2007} M. Ishiwata, On bounds for global solutions of semilinear parabolic equations with critical and subcritical Sobolev exponent, Differential Integral Equations, 20(2007), 1021--1034.
		
		\bibitem{Ikehata-Suzuki1996} R. Ikehata, T. Suzuki, Stable and unstable sets for evolution equations of parabolic and hyperbolic type, Hiroshima Math. J., 26(1996), 475--491.
		
		\bibitem{Ikehata-Suzuki} R. Ikehata and T. Suzuki, Semilinear parabolic equations involving critical Sobolev exponent: local and asymptotic behavior of solutions, Diff. Int. Eq., 13(2000), 869--901.
		
		\bibitem{Kaplan} S. Kaplan, On the growth of solutions of quasi-linear parabolic equations, Comm. Pure Appl. Math. 16(1963), 305--330.
		
		\bibitem{Lacey} A. A. Lacey, Thermal runaway in a non-local problem modelling Ohmic beating: Part 1: Model derivation and some special cases, European J. Appl. Math., 6(1995), 127--144.
		
		\bibitem{Levine} H.A. Levine, Some nonexistence and instability theorems for solutions of formally parabolic equations of the form $Pu_{t}=-Au+F(u)$, Arch. Rational Mech. Anal., 51(1973), 371--386.
		
		\bibitem{Lions1} J.-L. Lions, Quelques M\'ethodes de R\'esolution des Probl\`emes aux Limites Non Lin\'eaires, Dunod, Paris, (1969).
		
		\bibitem{Lions2} P.-L. Lions, Asymptotic behavior of some nonlinear heat equations, Phys. D 5(1982), 293--306.
		
		\bibitem{Li-Liu} X. Li, B. Liu, Finite time blow-up and global solutions for a nonlocal parabolic equation with Hartree type nonlinearity, Communications on Pure and Applied Analysis, 19(2020), 3093--3112.
		
		\bibitem{Li-Liu} X. Li, B. Liu, Vacuum isolating, blow up threshold and asymptotic behavior of solutions for a nonlocal parabolic equation, J. Math. Phys., 58(2017), 101503.
		
		\bibitem{Lieb-Loss} E. H. Lieb, M. Loss, Analysis, in: Graduate Studies in Mathematics, vol. 14, American Mathematical Society, Providence, RI, 4(2001).
		
		\bibitem{Liu-Ma}  B. Liu, L. Ma, Invariant sets and the blow up threshold for a nonlocal equation of parabolic type, Nonlinear Anal., 110(2014) 141--156.
		
		\bibitem{Li-Miao-Zhang} D. Li, C. Miao, X. Zhang, The focusing energy-critical Hartree equation, J. Differential Equations, 246(2009), 1139--1163.
		
		\bibitem{Lee-Ni}  T. Lee and W.-M. Ni, Global existence, large time behaviour and life span of solutions of a semilinear parabolic Cauchy problem, Trans. Amer. Math. Soc. 333(1992), 365--378.
		
		\bibitem{Levine-Payne} H.A. Levine and L.E. Payne, Some nonexistence theorems for initial-boundary value problems with nonlinear boundary constraints, Proc. Am. Math. Soc., 46(1974), 277--284.
		
		\bibitem{Lieberman} G.M. Lieberman, Second Order Parabolic Differential Equations, World Scientific Publishing, River Edge, 1996.
		
		\bibitem{Miao-Wu-Xu} C. Miao, Y. Wu, G. Xu, Dynamics for the focusing, energy-critical nonlinear Hartree equation, Forum Mathematicum, 27(2015), 373--447.
		
		\bibitem{Miao-Xu-Zhao1} C. Miao, G. Xu, L. Zhao, Global well-posedness, scattering and blow-up for the energy-critical, focusing Hartree equation in the radial case, Colloq. Math., 114(2009), 213--236.
		
		\bibitem{Miao-Xu-Zhao2} C. Miao, G. Xu, L. Zhao, On the blow up phenomenon for the mass critical focusing Hartree equation in $\mathbb{R}^{4}$, Colloq. Math., 119(2010), 23--50.
		
		\bibitem{Nakao-Ono} M. Nakao, K. Ono, Existence of global solutions to the Cauchy problem for the semilinear dissipative wave equations, Math. Z., 214(1993), 325--342.
		
		\bibitem{Ni-Sacks-Tavantzis} W. M. Ni, P. E. Sacks, J. Tavantzis, On the asymptotic behavior of solutions of certain quasilinear equations of parabolic type, J. Differential Equations., 54(1984), 97--120.
		
		\bibitem{Otani} M. Otani, Existence and asymptotic stability of strong solutions of nonlinear evolution equations with a difference term of subdifferentials, Colloq. Math. Soc. Janos Bolyai, Qualitative Theory of Differential Equations, 30(1980), North-Holland, Amsterdam.
		
		\bibitem{Ou-Wu} C. Ou, J. Wu, Persistence of wave fronts in delayed nonlocal reaction diffusion equations, J. Differential Equations, 235(2007), 219--261.
		
		\bibitem{Payne-Sattinger} L.E. Payne, D.H. Sattinger, Saddle points and instability of nonlinear hyperbolic equations, Israel J. Math., 22(1975), 273--303.
		
		\bibitem{Quittner1} P. Quittner, Universal bound for global positive solutions of a superlinear parabolic problem, Math. Ann., 320(2001), 299--305.
		
		\bibitem{Quittner2} P. Quittner, A priori bounds for global solutions of a semilinear parabolic problem, Acta Math. Univ. Comenian (N.S.), 68(1999), 195--203.
		
		\bibitem{Quittner-Souplet} P. Quittner, Ph. Souplet, Superlinear parabolic problems. Blow-up, global existence and steady states. Birkh\"auser Advanced Texts. Birkh\"auser, Basel, 2007.
		
		\bibitem{Sattinger} D. H. Sattinger, On global solution of nonlinear hyperbolic equations, Arch. Rat. Mech. Anal., 30(1968), 148--172.
		
		\bibitem{Suzuki} T. Suzuki, Semilinear parabolic equation on bounded domain with critical Sobolev exponent, Indiana Univ. Math. J., 57(2008) 3365--3396.
		
		\bibitem{Tan} Z. Tan, Global solutions and blowup of semilinear heat equation with critical sobolev exponent, Commun Partial Differential Equations, 26(2001), 717--741.
		
		\bibitem{Tanabe} H. Tanabe, Equations of Evolution, Pitman, (1979).
		
		\bibitem{Tsutsumi1} M. Tsutsumi, On solutions of semilinear differential equations in a Hubert space, Math. Japon., 17(1972), 173--193.
		
		\bibitem{Tsutsumi2} M. Tsutsumi, Existence and nonexistence of global solutions for nonlinear parabolic equations, Publ. RIMS, Kyoto Univ., 8(1972/73), 211--229.
		
		\bibitem{Weissler1} F.B. Weissler, Semilinear evolution equations in Banach spaces, J. Funct. Anal., 32(1979), 277--296.
		
		\bibitem{Weissler2}  F.B. Weissler, Local existence and nonexistence for semilinear parabolic equations in $L^{p}$, Indiana Univ. Math. J., 29(1980), 79--102.
		
		\bibitem{Zhang-Yang} J. Zhang, V. D. Rădulescu, M. Yang, et al. Global existence and blow-up solutions for a parabolic equation with critical nonlocal interactions,  J. Dyn. Diff. Equat. (2023). {\tt https://doi.org/10.1007/s10884-023-10278-y}.
		
		
	\end{thebibliography}
	
\end{document}